\documentclass[a4paper, 10pt]{article}
\usepackage[left=3cm,right=3cm,top=3cm,bottom=3cm]{geometry}

\usepackage[labelformat=simple, labelfont=sf]{caption}
\usepackage{subcaption}

\usepackage{packageList}
\usepackage{macros}

\graphicspath{{./Figures/}}
\pgfplotsset{compat=1.11}
\newlength\fwidth

\title{Greedy multiscale kernel algorithms}

\author[1]{Sara Avesani \thanks{\href{mailto:sara.avesani@usi.ch}{sara.avesani@usi.ch}}}
\author[2]{Armin Iske \thanks{\href{mailto:armin.iske@uni-hamburg.de}{armin.iske@uni-hamburg.de}}}
\author[1]{Michael Multerer \thanks{\href{mailto:michael.multerer@usi.ch}{michael.multerer@usi.ch}}}
\author[3,4]{Tizian Wenzel \thanks{\href{mailto:wenzel@math.lmu.de}{wenzel@math.lmu.de}}}

\affil[1]{USI Lugano, Lugano, Switzerland}
\affil[2]{University of Hamburg, Hamburg, Germany}
\affil[3]{LMU Munich, Munich, Germany}
\affil[4]{Munich Center for Machine Learning, Munich, Germany}

\DeclareMathOperator{\spn}{span}
\DeclareMathOperator{\conv}{absco}

\DeclareMathOperator{\diam}{diam}

\newcommand{\bs}{\boldsymbol}
\newcommand{\Kmat}[2]{K_{#1,#2}}

\definecolor{darkgreen}{rgb}{0, 0.6, 0}

\begin{document}

\maketitle

\begin{abstract}
Multiscale kernel approximation and greedy kernel algorithms are two established tools for scattered data approximation.
We combine them into \emph{greedy multiscale kernel algorithms}, in which every level of a multiscale residual correction is constructed by a greedy kernel algorithm.
For the multiscale $f$-greedy algorithm we prove convergence rates in the number of selected centres which coincide (up to an arbitrarily small loss in the exponent) 
with the rates of the single-scale $f$-greedy algorithm.
The gain of the multiscale construction is in the stability, and we make this precise by a two-sided estimate for the condition number of kernel matrices with rescaled kernels.
The lower bound shows that any single-scale method, and in particular the $f$-greedy algorithm, pays for its convergence rate with condition numbers growing at least like a power of the number of centres.
In contrast, the multiscale scaling of the kernel keeps this lower bound fixed on every level without changing the rate.
The upper bound identifies the separation of the centres at the scale of the kernel as the quantity to control.
We give the operation counts of all the methods considered and compare the greedy multiscale algorithms numerically with single-scale greedy and classical multiscale methods, in terms of error, time and conditioning.
\end{abstract}

\section{Introduction}

Kernel-based approximation methods are well-established tools in scattered data approximation, numerical analysis, and machine learning \cite{wendland2005scattered,fasshauer2015kernel}.
They are built around a positive definite kernel $k$, which induces a reproducing kernel Hilbert space (RKHS) and interpolation schemes with a rich theoretical foundation.
The definitions and assumptions used throughout this work are collected in \Cref{sec:background}.

Among the many developments in kernel approximation, two particularly successful directions are greedy kernel methods \cite{wirtz2013vectorial,haasdonk2018greedy,wenzel2023analysis,santin2024optimality} and multiscale kernel methods \cite{floater1996multistep,gia2010multiscale,wendland2010multiscale}, the latter also known as multilevel, multistep, or multiresolution methods.

Greedy kernel methods originate from the thinning algorithms introduced in \cite{floater1998thinning}.
Their objective is to construct accurate approximations from a carefully selected subset of the data sites.
Over time, the focus shifted from point removal to adaptive point insertion, leading to a variety of greedy selection rules \cite{marchi2005optimal,muller2009komplexitat,wirtz2013vectorial}, accompanied by a theoretical framework covering stability \cite{wenzel2021novel} and convergence \cite{santin2017convergence,wenzel2023analysis,santin2024optimality}.
Greedy kernel methods are used in numerous applications, from surrogate modeling to uncertainty quantification and scientific computing \cite{haasdonk2018greedy,doeppel2024goal}.

Multiscale kernel methods trace their origins to the multistep interpolation schemes of \cite{floater1996multistep}.
Their central idea is to approximate a target function across several resolution levels by iteratively correcting the residual of the previous levels: starting from a coarse set of centres, progressively finer point sets and progressively narrower kernels are used.
The framework was originally motivated by compactly supported kernels \cite{wendland1995piecewise}, for which a scaling of the support proportional to the fill distance yields sparse and well-conditioned system matrices.
Convergence was established in \cite{gia2010multiscale,wendland2010multiscale} with later refinements in \cite{townsend2013multiscale,wendland2017multiscale}, and the development of adaptive multiscale schemes is an active topic, see \cite[Algorithm~3]{wendland2017multiscale} and the recent preprints \cite{gollwitzer2026nodal,lot2026adaptive,durst2026improved}.

Although greedy and multiscale kernel methods pursue different strategies, they share a common objective, namely the construction of efficient kernel approximations from large data sets.
Greedy methods achieve this through adaptive centre selection, multiscale methods through a hierarchy of resolutions.
In this paper we combine both ideas and introduce a class of greedy multiscale kernel algorithms in which every level of the multiscale hierarchy is built by a greedy algorithm.
Since the $P$-greedy algorithm produces asymptotically quasi-uniformly distributed centres \cite{wenzel2021novel},
the classical multiscale construction on quasi-uniform sets can be seen as a $P$-greedy multiscale algorithm\footnote{To be precise, \cite{wenzel2021novel} establishes quasi-uniformity for a fixed shape parameter.
It remains to check whether this result extends uniformly to a range of shape parameters.},
while target-dependent rules such as $f$-greedy yield genuinely adaptive multiscale approximations.

The message of this paper is that the multiscale structure improves the stability of the standard $f$-greedy algorithm without losing approximation order.
The contributions are the following.
\begin{enumerate}[label=(\roman*)]
\item For the multiscale $f$-greedy algorithm we prove convergence rates in the number of selected centres, see \Cref{th:main_theorem}.
The rate coincides, up to an arbitrarily small loss in the exponent, with the rate of the single-scale $f$-greedy algorithm of \cite{santin2024optimality}.
\item We prove a two-sided estimate for the condition number of kernel matrices with rescaled kernels, see \Cref{prop:stability} and \Cref{th:tradeoff}.
\item We combine the multiscale contraction lemma behind all convergence proofs, \Cref{lem:contraction}, with the stability estimate into one theorem, \Cref{th:stability_convergence}, valid for any construction of the levels.
It separates the two roles: the reduction of the residual on a level decides the convergence, the separation of its centres at the scale of the kernel decides the stability.
The classical multiscale method and the multiscale $f$-greedy algorithm are recovered from it as special cases, see \Cref{rem:stability_convergence_cases}.
\item We give the operation counts of the greedy multiscale algorithms and of the methods they are compared with, see \Cref{sec:cost}, and compare them numerically, in error, time and conditioning, with the single-scale $f$- and $P$-greedy algorithms and with the classical multiscale method with compactly supported kernels.
\end{enumerate}

The remainder of the paper is organized as follows.
\Cref{sec:background} collects the necessary background.
\Cref{sec:multiscale_greedy} introduces the greedy multiscale algorithms and their cost.
\Cref{sec:analysis} contains the analysis: the rates for multiscale $f$-greedy, the stability estimate, and the general theorem that ties the two together.
\Cref{sec:numerical_results} presents the numerical experiments and \Cref{sec:conclusion} concludes.

\section{Background}
\label{sec:background}
\subsection{Kernels, native spaces and Sobolev spaces}
\label{subsec:basics}

Throughout, $\Omega \subset \R^d$ is a bounded domain with Lipschitz boundary, and $k(x,z) = \Phi(x-z)$ is a translation invariant kernel with an even function $\Phi : \R^d \to \R$.
For $\delta \in (0,1]$ we consider the rescaled kernels
\begin{align}
\label{eq:rescaled_kernel}
\Phi_\delta(x)  \isdef \delta^{-d} \Phi \left( \frac{x}{\delta} \right), \qquad k_\delta(x,z)  \isdef \Phi_\delta(x-z),
\end{align}
whose Fourier transforms satisfy $\widehat{\Phi_\delta}(\omega) = \widehat{\Phi}(\delta \omega)$.
We refer to \cite{wenzel2026optimal} for a theoretical analysis of the asymptotically optimal choice of the scale parameter.
We make the following assumptions on $\Phi$.
\begin{enumerate}[label=(A\arabic*)]
\item \label{ass:integrable} $\Phi \in L^1(\R^d) \cap C(\R^d)$.
\item \label{ass:fourier} There are constants $0 < c_\Phi \leq C_\Phi$ and $\tau > d/2$ such that
\begin{align}
\label{eq:fourier_decay}
c_\Phi (1+\Vert \omega \Vert_2^2)^{-\tau}
\leq
\widehat{\Phi}(\omega)
\leq
C_\Phi (1+\Vert \omega \Vert_2^2)^{-\tau}
\qquad \text{for all } \omega \in \R^d .
\end{align}
\end{enumerate}
They make $k$ strictly positive definite and identify its native space with a Sobolev space, see below for details. In particular, they hold for the Mat\'ern kernels and for the compactly supported Wendland kernels \cite{wendland1995piecewise}.
For the rescaled kernels, \eqref{eq:fourier_decay} reads
\begin{align}
\label{eq:fourier_decay_scaled}
c_\Phi (1+\delta^2 \Vert \omega \Vert_2^2)^{-\tau}
\leq
\hat{\Phi}_\delta(\omega)
\leq C_\Phi (1+\delta^2 \Vert \omega \Vert_2^2)^{-\tau}
\qquad \text{for all } \omega \in \R^d,
\end{align}
and we recall the elementary estimate, valid for $\delta \in (0,1]$,
\begin{align}
\label{eq:equivalence_fourier_weights}
(1+\Vert \omega \Vert_2^2)^{-\tau}
\leq
(1+\delta^2 \Vert \omega \Vert_2^2)^{-\tau}
\leq
\delta^{-2\tau}
(1+\Vert \omega \Vert_2^2)^{-\tau}.
\end{align}
The kernel $k_\delta$ induces on $\R^d$ the RKHS $\calh_{k_\delta}(\R^d)$ with inner product
\begin{align*}
\langle f,g \rangle_{\Phi_\delta}
 \isdef (2\pi)^{-d/2}
\int_{\R^d}
\frac{\widehat{f}(\omega)\overline{\widehat{g}(\omega)}}
{\widehat{\Phi_\delta}(\omega)}
\,\mathrm{d}\omega ,
\end{align*}
and we abbreviate the norm of $\calh_{k_\delta}(\R^d)$ by $\Vert \cdot \Vert_{\Phi_\delta}$, to distinguish it from the norm of the native space $\calh_{k_\delta}(\Omega)$ on the domain.
The Sobolev space $H^\tau(\R^d)$ is normed by $\Vert g \Vert_{H^\tau(\R^d)}^2 = {(2\pi)^{-d/2}} \int_{\R^d} |\widehat g(\omega)|^2 (1+\Vert \omega \Vert_2^2)^\tau \,\mathrm{d}\omega$.
Combining \eqref{eq:fourier_decay_scaled} and \eqref{eq:equivalence_fourier_weights} yields the following norm equivalence, see \cite[Lemma~1]{wendland2010multiscale}.

\begin{lemma}
\label{lem:rkhs_norms_different_scalings}
For every $\delta \in (0,1]$ we have $\calh_{k_\delta}(\R^d) = H^\tau(\R^d)$ and every $g \in H^\tau(\R^d)$ satisfies
\begin{align}
\label{eq:rkhs_norms_different_scalings}
c_\Phi^{1/2} \Vert g \Vert_{\Phi_\delta}
\leq
\Vert g \Vert_{H^\tau(\R^d)}
\leq
C_\Phi^{1/2} \delta^{-\tau} \Vert g \Vert_{\Phi_\delta}.
\end{align}
\end{lemma}

On the bounded domain $\Omega$ the native space $\calh_{k_\delta}(\Omega)$ is the space of restrictions of functions of $\calh_{k_\delta}(\R^d)$, and its norm is the quotient norm.
We collect the facts we need, see \cite[Chapters 10 and 13]{wendland2005scattered}.
For a finite set $Z \subset \Omega$ we denote by $I_{\delta,Z} g$ the kernel interpolant of the function $g$ in the points $Z$ with respect to the kernel $k_\delta$,
i.e.\ the unique element of $V_{\delta,Z}  \isdef \spn \{ k_\delta( \cdot ,z) : z \in Z\}$ with $(I_{\delta,Z} g)(z) = g(z)$ for $z \in Z$.
The kernel matrix is denoted by $\Kmat{\delta}{Z}  \isdef (k_\delta(z_i,z_j))_{i,j}$.

\begin{lemma}
\label{lem:native_space_omega}
Let $\delta \in (0,1]$ and let $Z \subset \Omega$ be finite.
\begin{enumerate}[label=(\roman*)]
\item $\calh_{k_\delta}(\Omega) = \{ G|_\Omega : G \in \calh_{k_\delta}(\R^d) \}$ and $\Vert g \Vert_{\calh_{k_\delta}(\Omega)} = \min \{ \Vert G \Vert_{\Phi_\delta} : G|_\Omega = g \}$.
In particular $\Vert G|_\Omega \Vert_{\calh_{k_\delta}(\Omega)} \leq \Vert G \Vert_{\Phi_\delta}$ for every $G \in \calh_{k_\delta}(\R^d)$.
\item For $g \in \calh_{k_\delta}(\Omega)$, the interpolant $I_{\delta,Z} g$ is the orthogonal projection of $g$ onto $V_{\delta,Z}$ in $\calh_{k_\delta}(\Omega)$.
Consequently $\Vert I_{\delta,Z} g \Vert_{\calh_{k_\delta}(\Omega)} \leq \Vert g \Vert_{\calh_{k_\delta}(\Omega)}$ and $\Vert g - I_{\delta,Z} g \Vert_{\calh_{k_\delta}(\Omega)} \leq \Vert g \Vert_{\calh_{k_\delta}(\Omega)}$, and the same statements hold in $\calh_{k_\delta}(\R^d)$.
Moreover $(I_{\delta,Z} G)|_\Omega = I_{\delta,Z} (G|_\Omega)$ for $G \in \calh_{k_\delta}(\R^d)$, since $Z \subset \Omega$.
\item The norm of the interpolant is computable from the linear algebra of the interpolation problem,
$\Vert I_{\delta,Z} g \Vert_{\calh_{k_\delta}(\Omega)}^2 = g|_Z^\top \Kmat{\delta}{Z}^{-1} g|_Z$.
\end{enumerate}
\end{lemma}

Since $\Omega$ has a Lipschitz boundary, Sobolev spaces on $\Omega$ and on $\R^d$ are related by a universal extension operator, see \cite[Proposition~1]{wendland2010multiscale}.

\begin{prop}
\label{prop:universal_ext_op}
There is a linear operator $E : H^\sigma (\Omega) \to H^\sigma (\R^d)$, defined for all $\sigma \geq 0$ simultaneously, such that for every $\sigma \geq 0$ and every $f \in H^\sigma(\Omega)$
\begin{enumerate}[label=(\roman*)]
    \item $Ef|_{\Omega} = f$,
    \item $\|Ef\|_{H^\sigma (\R^d)} \leq C_{\sigma} \|f\|_{H^\sigma (\Omega)}$.
\end{enumerate}
\end{prop}

By \Cref{lem:rkhs_norms_different_scalings}, \Cref{lem:native_space_omega}(i) and \Cref{prop:universal_ext_op}, every $g \in H^\tau(\Omega)$ belongs to $\calh_{k_\delta}(\Omega)$ for every $\delta \in (0,1]$, with
\begin{align}
\label{eq:Htau_to_native}
\Vert g \Vert_{\calh_{k_\delta}(\Omega)} \leq \Vert E g \Vert_{\Phi_\delta} \leq c_\Phi^{-1/2} \Vert Eg \Vert_{H^\tau(\R^d)} \leq c_\Phi^{-1/2} C_\tau \Vert g \Vert_{H^\tau(\Omega)} .
\end{align}

Finally, we need a sampling inequality.
The following result is taken from \cite{wendland2005approximate}, see also \cite{narcowich2006sobolev}.
For $X \subset \Omega$ we denote the fill distance and the separation distance by
\begin{align*}
h_{X,\Omega}  \isdef \sup_{x \in \Omega} \min_{x_i \in X} \Vert x - x_i \Vert_2 ,
\qquad
q_X  \isdef \tfrac12 \min_{x_i \neq x_j \in X} \Vert x_i - x_j \Vert_2 .
\end{align*}
\begin{theorem}
\label{th:sampling_ineq}
Let $\tau > d/2$.
There are constants $C > 0$ and $h_0 > 0$, depending only on $\Omega$ and $\tau$, such that for every finite $X \subset \Omega$ with $h_{X,\Omega} \leq h_0$ and every $u \in H^\tau(\Omega)$ it holds:
\begin{equation*}
\Vert u \Vert_{L^\infty(\Omega)} \leq C \left( h_{X,\Omega}^{\tau - d/2} \Vert u \Vert_{H^\tau(\Omega)} + \Vert u|_X \Vert_\infty \right),
\end{equation*}
and 
\begin{equation*}
\Vert u \Vert_{L^2(\Omega)} \leq C \left( h_{X,\Omega}^{\tau} \Vert u \Vert_{H^\tau(\Omega)} + \Vert u|_X \Vert_\infty \right).
\end{equation*}
\end{theorem}
If, in addition, $u|_X = 0$, the second terms vanish, and $\Vert u \Vert_{L^p(\Omega)} \leq \Vert u \Vert_{L^2(\Omega)}^{2/p} \Vert u \Vert_{L^\infty(\Omega)}^{1-2/p}$ gives, for $2 \leq p \leq \infty$,
\begin{align}
\label{eq:sampling_Lp}
\Vert u \Vert_{L^p(\Omega)} \leq C h_{X,\Omega}^{\tau - d(1/2 - 1/p)} \Vert u \Vert_{H^\tau(\Omega)}.
\end{align}

\subsection{Multiscale kernel approximation}
\label{subsec:multiscale_kernel}

Multiscale kernel methods, introduced in \cite{floater1996multistep}, approximate a target function by a residual correction over a sequence of scales.
Given point sets $X_1, X_2, \ldots \subset \Omega$ with decreasing fill distances $h_j  \isdef h_{X_j,\Omega}$ and kernels $\Phi_j  \isdef \Phi_{\delta_j}$ with scales $\delta_j$ proportional to $h_j$, the approximation is built by \Cref{alg:standard_multiscale_algorithm}, where the local approximant on level $j$ is the kernel interpolant $s_j = I_{\delta_j, X_j} e_{j-1}$.

\begin{algorithm}
\caption{Multiscale approximation \cite{wendland2010multiscale}}
\label{alg:standard_multiscale_algorithm}
\textbf{Input}: target $f$, point sets $X_1, \dots, X_n$, scales $\delta_1, \dots, \delta_n$ \\
\textbf{Output}: approximation $f_n \in V_{\delta_1,X_1} + \dots + V_{\delta_n,X_n}$
\begin{algorithmic}[1]
\State Set $f_0 = 0$, $e_0 = f$
\For{$j = 1, 2, \dots, n$}
    \State $s_j = I_{\delta_j, X_j} e_{j-1}$
    \State $f_j = f_{j-1} + s_j$, \quad $e_j = e_{j-1} - s_j$
\EndFor
\end{algorithmic}
\end{algorithm}

The convergence of \Cref{alg:standard_multiscale_algorithm} was established in \cite{wendland2010multiscale}.

\begin{theorem}[{\cite[Theorem~1]{wendland2010multiscale}}]
\label{th:conv_result_wendland}
Let the fill distances satisfy $c\mu h_j \leq h_{j+1} \leq \mu h_j$ for $j = 1,2,\ldots$ with fixed $\mu \in (0,1)$, $c \in (0,1]$ and $h_1$ sufficiently small, let $\delta_j = \nu h_j$ with $1/h_1 \geq \nu \geq \gamma/\mu$ for a fixed $\gamma > 0$, and let $f \in H^\tau(\Omega)$.
Then there is a constant $C_1 = C_1(\gamma) > 0$ such that, with $\alpha = C_1 \mu^\tau$,
\begin{align*}
\|E e_j\|_{\Phi_{j+1}}
\leq \alpha \|E e_{j-1}\|_{\Phi_j}
\qquad \text{for } j = 1,2,\ldots ,
\end{align*}
and hence there is a constant $C > 0$ such that
$\|f - f_n\|_{L^2(\Omega)} \leq C \alpha^n \|f\|_{H^\tau(\Omega)}$ for all $n$.
In particular $f_n \to f$ linearly in $L^2(\Omega)$ if $\alpha < 1$.
\end{theorem}

The proof combines the sampling inequality on each level with the norm equivalence of \Cref{lem:rkhs_norms_different_scalings}, which relates the residuals across levels despite the change of scale.
We isolate this mechanism in \Cref{lem:contraction} below.
Since $h_{n} \leq \mu^{n-1} h_1$, the estimate can be rewritten in terms of the fill distance of the finest level, and one recovers, up to constants and an arbitrarily small loss in the exponent, the approximation order of single-scale kernel interpolation with fill distance $h_n$, see \Cref{th:multiscale_using_fill_dist}.
The multiscale construction therefore does not improve the approximation order.
Its advantages lie elsewhere: the scale $\delta_j \asymp h_j$ keeps the ratio $\delta_j / q_{X_j}$ bounded for quasi-uniform sets, which is exactly what keeps the level matrices well conditioned, see \Cref{prop:stability}, and for compactly supported kernels it makes the level matrices sparse and the evaluation local.

\subsection{Greedy kernel algorithms}
\label{subsec:greedy_kernel_algorithm}

Greedy kernel algorithms build interpolation sets one point at a time.
Starting from $X_0  \isdef \emptyset$, the next point is $x_{n+1}  \isdef \argmax_{x \in \Omega} \eta_n(x)$ for an error indicator $\eta_n : \Omega \to \R$.
In practice the maximization runs over a large but finite candidate set $X_N \subset \Omega$.
The $f$-greedy algorithm \cite{schaback2000adaptive}, the $P$-greedy algorithm \cite{marchi2005optimal}, and the hybrids $f/P$-greedy \cite{muller2009komplexitat} and $f \cdot  P$-greedy \cite{dutta2020greedy} are unified by the $\beta$-greedy rule of \cite{wenzel2023analysis},
\begin{align}
\label{eq:beta_greedy}
x_{n+1}  \isdef \argmax_{x\in\Omega\setminus X_n} |(f-s_n)(x)|^\beta   P_n(x)^{1-\beta},
\end{align}
where $s_n = I_{\delta,X_n} f$ and $P_n$ is the power function of $X_n$.
The choice $\beta = 1$ is $f$-greedy and $\beta = 0$ is $P$-greedy.
A convergence theory for $\beta$-greedy algorithms was established in \cite{wenzel2023analysis} and refined in \cite{santin2024optimality}; convergence under mild conditions is studied in \cite{albrecht2025convergence}.
We recall the estimate we build on.

\begin{cor}[{\cite[Corollary~3.9]{santin2024optimality} for $\gamma=1$}]
\label{cor:cor_from_fgreedy_paper}
Let $\Omega \subset \R^d$ have a Lipschitz boundary, let $\tau > d/2$, and assume that $\calh_k(\Omega)$ is continuously embedded into $H^\tau(\Omega)$.
Then any $\beta$-greedy algorithm \eqref{eq:beta_greedy} satisfies for every $f \in \calh_k(\Omega)$ and every $n \in \N$
\begin{align*}
\min_{n+1 \leq i \leq 2n} \|f - I_{X_i} f\|_{L^\infty(\Omega)} \leq c_\beta  \left( n^{-\frac{\tau}{d} + \frac{1-\beta}{2}} \right)^{\frac{1}{\max(1, \beta)}} \|f - I_{X_{n+1}} f\|_{\calh_k(\Omega)} ,
\end{align*}
with a constant $c_\beta > 0$.
\end{cor}
Together with the monotonicity of the RKHS interpolation error this gives, see \cite[Eq.~(3.9)]{santin2024optimality},
\begin{align}
\label{eq:beta_greedy_error}
\min_{1 \leq i \leq m} \|f - I_{X_i} f\|_{L^\infty(\Omega)} \leq c_\beta''
\left( m^{-\frac{\tau}{d} + \frac{1 - \beta}{2}} \right)^{\frac{1}{\max(1, \beta)}}
\|f\|_{\calh_k(\Omega)} .
\end{align}
For $f$-greedy, $\beta = 1$, the rate is $m^{-\tau/d}$.
For quasi-uniform centres the sampling inequality gives 
$$h^{\tau - d/2} \asymp m^{-\tau/d + 1/2}$$
in $L^\infty(\Omega)$, so the $f$-greedy rate is better by exactly the factor $m^{1/2}$, independently of $\tau$ and $d$.
This gap is the benefit of target-dependent selection that we want to extend to the multiscale setting.

\section{Greedy multiscale algorithms}
\label{sec:multiscale_greedy}
\subsection{The algorithm}

Our goal is to combine the adaptive centre selection of greedy kernel methods with the hierarchical residual correction of multiscale kernel methods.
Instead of prescribing quasi-uniform point sets, we run a greedy algorithm on every level of the hierarchy: level $\ell$ approximates the residual $e_{\ell-1}$ of the previous levels with the kernel $k_{\delta_\ell}$, and the centres of the level are chosen by the greedy rule applied to $e_{\ell-1}$.
The procedure is summarized in \Cref{alg:multiscale_greedy}.

\begin{algorithm}
\caption{Greedy multiscale approximation}
\label{alg:multiscale_greedy}
\textbf{Input}: target $f$, scales $\delta_1 > \delta_2 > \dots > \delta_n$, level sizes $M_1, \dots, M_n$ \\
\textbf{Output}: approximation $f_n = s_1 + \dots + s_n$
\begin{algorithmic}[1]
\State Set $f_0 = 0$, $e_0 = f$
\For{$\ell = 1, 2, \dots, n$}
    \State Run the greedy algorithm \eqref{eq:beta_greedy} on $e_{\ell-1}$ with kernel $k_{\delta_\ell}$ for $M_\ell$ steps, producing $X_1 \subset \dots \subset X_{M_\ell}$
    \State Let $Z_\ell  \isdef X_{i^*}$ with $i^*  \isdef \argmin_{1 \leq i \leq M_\ell} \Vert e_{\ell-1} - I_{\delta_\ell, X_i} e_{\ell-1} \Vert_{L^\infty(\Omega)}$ \label{line:best_iterate}
    \State $s_\ell = I_{\delta_\ell, Z_\ell} e_{\ell-1}$
    \State $f_\ell = f_{\ell-1} + s_\ell$, \quad $e_\ell = e_{\ell-1} - s_\ell$
\EndFor
\end{algorithmic}
\end{algorithm}

We recall that, in practice, $\Omega$ is replaced by a finite candidate set $X_N$, the maxima are taken over $X_N$, and the residual is only known on $X_N$.

The level sizes are the analogue, in terms of numbers of points, of the classical choice of mesh size and kernel scale \(h_\ell = \mu h_{\ell-1}\), \(\delta_\ell = \nu h_\ell\) from \Cref{th:conv_result_wendland}. For quasi-uniform point sets, the number of points satisfies \(M_\ell \asymp h_\ell^{-d}\), so the geometric reduction of the fill distance translates into a geometric growth of the cardinalities. In terms of the number of points, this becomes
\begin{align}
\label{eq:level_cardinality}
M_{\ell} = \mu^{-d} M_{\ell-1}, \qquad \delta_\ell = \nu M_\ell^{-1/d},
\end{align}
with \(\mu \in (0,1)\) and \(\nu \geq 1\). 
\subsection{Relation to the classical multiscale method}
\label{subsec:relation_classical}

The $P$-greedy algorithm selects centres that become asymptotically quasi-uniformly distributed, see for instance \cite{wenzel2021novel}, so the multiscale $P$-greedy algorithm produces, level by level, the quasi-uniform sets on which the classical construction of \Cref{alg:standard_multiscale_algorithm} is built.
In this sense the classical multiscale method can be understood as \Cref{alg:multiscale_greedy} with the $P$-greedy rule.
We note, however, that this is not an argument in favour of multiscale $P$-greedy as a method:
a quasi-uniform subset of the data can be produced by geometric thinning at a cost that is linear, up to logarithms, in the number of data points, whereas a greedy selection costs $\mathcal{O}(N M^2)$ for $M$ centres out of $N$ candidates, see \Cref{sec:cost}.
The interesting members of the family are the target-dependent rules, $\beta > 0$, and in particular $f$-greedy, for which the level sets adapt to the residual and the classical assumption of quasi-uniformity is given up.
We also recall that nested level sets, required by the original construction of \cite{floater1996multistep}, are not needed for the analysis.
\subsection{Computational cost}
\label{sec:cost}

Let $N$ be the number of candidates, $N'$ the number of evaluation points, $M_\ell$ the number of centres of level $\ell$ and $N_n  \isdef \sum_{\ell \leq n} M_\ell$ the total number of centres.
A greedy step on a level with $M$ current centres updates the residual, or the Newton basis, on all candidates, at a cost of $\mathcal{O}(NM)$ operations, and performs $\mathcal{O}(M^2)$ operations of linear algebra for the incremental factorization. The memory is $\mathcal{O}(NM)$ for the kernel columns on the candidates.
A level with $M_\ell$ centres therefore costs $\mathcal{O}(N M_\ell^2)$, and with the growth $M_{\ell+1} \geq \mu^{-d} M_\ell$ the whole hierarchy costs
\begin{align*}
\sum_{\ell=1}^n \mathcal{O}(N M_\ell^2) = \mathcal{O}\bigg( \frac{N M_n^2}{1 - \mu^{2d}} \bigg) .
\end{align*}
The single-scale $f$-greedy algorithm with the same total number of centres $N_n \leq M_n/(1-\mu^d)$ costs $\mathcal{O}(N N_n^2)$.
The two counts are of the same order, and their ratio is bounded by the constant $(1+\mu^d)/(1-\mu^d)$, which is $5/3$ for $d=2$ and $\mu = 1/2$.
The multiscale structure therefore does not reduce the cost of the greedy selection.
The evaluation of the final approximant at $N'$ points costs $\mathcal{O}(N' N_n)$ for globally supported kernels, and $\mathcal{O}(N' n)$ for compactly supported kernels with $\delta_\ell \asymp h_\ell$, since every evaluation point then interacts with a bounded number of centres on each of the $n$ levels.

For comparison, the classical \Cref{alg:standard_multiscale_algorithm} on quasi-uniform sets solves one linear system per level, whose condition number is bounded independently of the level, see \cite{wendland2010multiscale, avesani2025multiscale}.
With a compactly supported kernel the matrix has a bounded number of nonzeros per row and a level costs $\mathcal{O}(M_\ell)$ operations.
With a globally supported kernel the level matrix is compressed in samplet coordinates, which leaves $\mathcal{O}(M_\ell \log M_\ell)$ entries \cite{avesani2025multiscale}, so that the assembly, every matrix-vector product and hence the whole level cost $\mathcal{O}(M_\ell \log M_\ell)$ operations. The transfer of the residual to the next level uses the same compression.
The quasi-uniform hierarchy itself is obtained from the data by cell-based thinning in $\mathcal{O}(N)$ operations per level.
Summed over the levels, the classical method costs $\mathcal{O}(N n + N_n)$ operations with compactly and $\mathcal{O}(N n + N_n \log N_n)$ with globally supported kernels, against $\mathcal{O}(N M_n^2)$ for the greedy selection.
The evaluation of a globally supported expansion at $N'$ points can be compressed in the same way for every method. We report the dense count $\mathcal{O}(N' N_n)$ for the greedy methods, which is what the implementation does.
\Cref{tab:cost} summarizes the counts.

\begin{table}[htb]
\centering
\small
\renewcommand{\arraystretch}{1.25}
\setlength{\tabcolsep}{4pt}
\begin{tabular}{l c c}
\hline
 & selection and solves & evaluation \\
\hline
single-scale $f$- or $P$-greedy, $N_n$ centres & $\mathcal{O}(N N_n^2)$ & $\mathcal{O}(N' N_n)$  \\
multiscale $f$- or $P$-greedy, \Cref{alg:multiscale_greedy} & $\mathcal{O}(N M_n^2)$ & $\mathcal{O}(N' N_n)$  \\
multiscale globally supported kernels, \Cref{alg:standard_multiscale_algorithm} & $\mathcal{O}(N n + N_n \log N_n)$ & $\mathcal{O}((N' + N_n) \log N')$  \\
multiscale compactly supported kernels, \Cref{alg:standard_multiscale_algorithm} & $\mathcal{O}(N n + N_n)$ & $\mathcal{O}(N' n)$ \\
\hline
\end{tabular}
\caption{Operation counts for $N$ candidates, $n$ levels with $M_\ell$ centres, $N_n = \sum_\ell M_\ell \asymp M_n$ centres in total, and $N'$ evaluation points.}
\label{tab:cost}
\end{table}
\section{Analysis}
\label{sec:analysis}

In \Cref{subsec:analysis_rescaled} we track the dependence on the scale $\delta$ in the convergence analysis of the $f$-greedy algorithm.
In \Cref{subsec:contraction} we isolate the core mechanism of the multiscale analysis in a contraction lemma, namely \Cref{lem:contraction}. Starting from this general lemma, in \Cref{subsec:error_bounds_standard} and \Cref{subsec:error_bounds_greedy_multiscale} we deduce the convergence of the classical multiscale method and of the multiscale $f$-greedy algorithm.
In \Cref{subsec:stability} we prove the two-sided estimate for the condition numbers of the level matrices, see \Cref{prop:stability} and \Cref{th:tradeoff}. Its lower bound shows that no single-scale method can be guaranteed to be accurate without becoming ill conditioned.
Finally, \Cref{subsec:stability_convergence} combines the two ingredients into a single statement, \Cref{th:stability_convergence},
valid for any construction of the levels: the reduction of the residual on a level governs the convergence,
and the ratio between the kernel scale and the separation of its centres governs the stability.
The classical multiscale method and the multiscale $f$-greedy algorithm are two cases of \Cref{th:stability_convergence}, the first controlling the separation of the centres and the second the reduction of the residual.
\subsection{\texorpdfstring{The $f$-greedy algorithm with rescaled kernels}{The f-greedy algorithm with rescaled kernels}}
\label{subsec:analysis_rescaled}

We first extend the crucial steps of the $f$-greedy analysis of \cite{wenzel2023analysis,santin2024optimality} to the rescaled kernel $k_\delta$, keeping track of the constants.

\begin{prop}
\label{prop:eigenvalue_equivalence}
Let $\Phi$ satisfy \ref{ass:fourier} and let $\delta \in (0,1]$.
Then for every finite $X \subset \R^d$ the eigenvalues of the kernel matrices satisfy
\begin{align}
\label{eq:equivalence_eigenvalues}
\frac{c_\Phi}{C_\Phi}   \lambda_j(\Kmat{1}{X}) \leq \lambda_j(\Kmat{\delta}{X}) \leq \frac{C_\Phi}{c_\Phi} \delta^{-2\tau}   \lambda_j(\Kmat{1}{X}), \qquad j = 1, \dots, |X|.
\end{align}
\end{prop}

\begin{proof}
By combining \eqref{eq:fourier_decay}, \eqref{eq:fourier_decay_scaled} and \eqref{eq:equivalence_fourier_weights},
we obtain
\begin{align*}
\frac{c_\Phi}{C_\Phi} \cdot \hat{\Phi}(\omega) \leq \hat{\Phi}_\delta(\omega) \leq \frac{C_\Phi}{c_\Phi} \delta^{-2\tau} \cdot \hat{\Phi}(\omega).
\end{align*}
Next we recall the well-known identity
\begin{align*}
\alpha^\top k(X, X) \alpha = (2\pi)^{-d/2} \cdot \int_{\R^d} \hat{\Phi}(\omega) \cdot \left| \sum_{j=1}^{|X|} \alpha_j e^{i \omega^\top x_j} \right|^2 ~ \mathrm{d}\omega.
\end{align*}
Combining the previous two (in)equalities yields
\begin{align*}
\frac{c_\Phi}{C_\Phi} \alpha^\top k(X, X)\alpha \leq \alpha^\top k_\delta(X, X) \alpha \leq \frac{C_\Phi}{c_\Phi} \delta^{-2\tau} \cdot \alpha^\top k(X, X)\alpha.
\end{align*}
The result then follows by applying the Rayleigh quotient characterization of the eigenvalues.
\end{proof}

For a sequence of points $x_1, x_2, \ldots$ we write $X_i  \isdef \{x_1, \dots, x_i\}$, $X_0  \isdef \emptyset$, and $P_{\delta,i}$ for the power function of $X_i$ with respect to $k_\delta$.

\begin{prop}
\label{prop:geom_mean_powerfunc_rescaled}
Let $\Omega$ have a sufficiently smooth boundary, let $\Phi$ satisfy \ref{ass:fourier} and let $\delta \in (0,1]$.
Then there is a constant $c > 0$, independent of $\delta$, such that for every sequence $\{ x_i \}_{i \geq 1} \subset \Omega$ of pairwise distinct points and every $n \in \N$
\begin{align*}
\left( \prod_{i=0}^{n-1} P_{\delta, i}(x_{i+1}) \right)^{1/n} \leq c \delta^{-\tau} n^{-\tau / d + 1/2}.
\end{align*}
\end{prop}

\begin{proof}
As in \cite[Proof of Proposition 4.1]{santin2024optimality}, the product of the power function values is the product of the diagonal entries of the Cholesky factor of $\Kmat{\delta}{X_n}$, so that $\prod_{i=0}^{n-1} P_{\delta,i}(x_{i+1})^2 = \det \Kmat{\delta}{X_n} = \prod_{j=1}^n \lambda_j(\Kmat{\delta}{X_n})$.
By \eqref{eq:equivalence_eigenvalues},
\begin{align*}
\left( \prod_{i=0}^{n-1} P_{\delta, i}(x_{i+1}) \right)^{1/n}
&= \left( \prod_{j=1}^{n} \lambda_j(\Kmat{\delta}{X_n}) \right)^{1/(2n)}
\leq \sqrt{\frac{C_\Phi}{c_\Phi}} \delta^{-\tau} \left( \prod_{j=1}^n \lambda_j(\Kmat{1}{X_n}) \right)^{1/(2n)} \\
&= \sqrt{\frac{C_\Phi}{c_\Phi}} \delta^{-\tau} \left( \prod_{i=0}^{n-1} P_{1, i}(x_{i+1}) \right)^{1/n}.
\end{align*}
The last factor no longer depends on $\delta$ and is estimated as in \cite[Proof of Theorem 3.4]{santin2024optimality}:
with $v_i  \isdef k( \cdot , x_i)$ and $V_i  \isdef \spn \{ v_j : 1 \leq j \leq i \}$, \cite[Lemma 2.1 and Remark 2.3]{santin2024optimality} give
$( \prod_{i=0}^{n-1} P_{1, i}(x_{i+1}) )^{1/n} \leq \sqrt{5} n^{1/2} \varepsilon_n$,
where $\varepsilon_n$ is the dyadic entropy number of $\conv(\{k( \cdot ,x) : x \in \Omega\})$ in $\ns$, and \cite[Section 3.4]{santin2024optimality} shows $\varepsilon_n \leq C n^{-\tau/d}$ because $\ns \hookrightarrow W_2^\tau(\Omega)$ by \ref{ass:fourier}.
The claim follows with $c = \sqrt{5} C \sqrt{C_\Phi / c_\Phi}$.
\end{proof}

\begin{cor}
\label{cor:f_greedy_rescaled_kernel}
Under the assumptions of \Cref{prop:geom_mean_powerfunc_rescaled}, the $f$-greedy algorithm with kernel $k_\delta$ satisfies for every $f \in \calh_{k_\delta}(\Omega)$ and every $n \in \N$
\begin{align*}
\min_{0 \leq i \leq n-1} \Vert f - I_{\delta, X_i} f \Vert_{L^\infty(\Omega)}
\leq c \delta^{-\tau} n^{-\tau/d} \Vert f \Vert_{\calh_{k_\delta}(\Omega)} ,
\end{align*}
with $c$ as in \Cref{prop:geom_mean_powerfunc_rescaled}.
\end{cor}

\begin{proof}
We follow \cite[Lemma 6]{wenzel2023analysis} and \cite[Proposition 3.3]{santin2024optimality}.
Let $r_{\delta, i}  \isdef f - I_{\delta, X_i} f$.
The Newton basis expansion of the interpolant gives $\sum_{i=0}^{n-1} ( r_{\delta,i}(x_{i+1}) / P_{\delta,i}(x_{i+1}) )^2 = \Vert I_{\delta,X_n} f \Vert_{\calh_{k_\delta}(\Omega)}^2 \leq \Vert f \Vert_{\calh_{k_\delta}(\Omega)}^2$, see \cite[Eq.~(7)]{wenzel2023analysis} and \Cref{lem:native_space_omega}(ii).
By the inequality between geometric and arithmetic means,
\begin{align*}
\left[ \prod_{i=0}^{n-1} \left( \frac{r_{\delta, i}(x_{i+1})}{P_{\delta, i}(x_{i+1})} \right)^2 \right]^{1/n}
\leq \frac{1}{n} \sum_{i=0}^{n-1} \left( \frac{r_{\delta, i}(x_{i+1})}{P_{\delta, i}(x_{i+1})} \right)^2
\leq \frac{1}{n} \Vert f \Vert_{\calh_{k_\delta}(\Omega)}^2 ,
\end{align*}
hence
$( \prod_{i=0}^{n-1} |r_{\delta, i}(x_{i+1})| )^{1/n} \leq n^{-1/2} ( \prod_{i=0}^{n-1} P_{\delta, i}(x_{i+1}) )^{1/n} \Vert f \Vert_{\calh_{k_\delta}(\Omega)}$.
The $f$-greedy selection gives $|r_{\delta, i}(x_{i+1})| = \Vert r_{\delta, i} \Vert_{L^\infty(\Omega)}$, the minimum of $n$ numbers is at most their geometric mean, and \Cref{prop:geom_mean_powerfunc_rescaled} bounds the geometric mean of the power function values.
\end{proof}

The meaning of \emph{sufficiently smooth boundary} in \Cref{prop:geom_mean_powerfunc_rescaled} and \Cref{cor:f_greedy_rescaled_kernel} is the one required for the entropy number bounds of \cite{santin2024optimality}. We refer to \cite[Remark 3.6]{santin2024optimality} for a discussion.
\subsection{The multiscale contraction lemma}
\label{subsec:contraction}

The convergence proofs of \cite{wendland2010multiscale,wendland2017multiscale} and the ones below share one mechanism: an $L^2(\Omega)$ estimate of the residual of a level, measured against the RKHS norm of the level's input, is turned into a contraction of the RKHS norms across the change of scale.
We formulate this mechanism once.
The first ingredient is the following $H^\tau(\Omega)$ estimate for kernel interpolation, which is \cite[Eq.~(24)]{wendland2017multiscale}.

\begin{lemma}
\label{lem:interp_error_Htau}
Let $g \in H^\tau(\Omega)$, $\delta \in (0,1]$ and let $Z \subset \Omega$ be finite.
Then
\begin{align*}
\Vert g - I_{\delta,Z} g \Vert_{H^\tau(\Omega)} \leq C_\Phi^{1/2} \delta^{-\tau} \Vert E g \Vert_{\Phi_\delta} .
\end{align*}
\end{lemma}

\begin{proof}
Let $G  \isdef Eg \in H^\tau(\R^d) = \calh_{k_\delta}(\R^d)$.
By \Cref{lem:native_space_omega}(ii), $(G - I_{\delta,Z} G)|_\Omega = g - I_{\delta,Z} g$, and since restriction to $\Omega$ does not increase the $H^\tau$ norm,
\begin{align*}
\Vert g - I_{\delta,Z} g \Vert_{H^\tau(\Omega)}
\leq \Vert G - I_{\delta,Z} G \Vert_{H^\tau(\R^d)}
\leq C_\Phi^{1/2} \delta^{-\tau} \Vert G - I_{\delta,Z} G \Vert_{\Phi_\delta}
\leq C_\Phi^{1/2} \delta^{-\tau} \Vert G \Vert_{\Phi_\delta} ,
\end{align*}
by \eqref{eq:rkhs_norms_different_scalings} and the projection property of \Cref{lem:native_space_omega}(ii) in $\calh_{k_\delta}(\R^d)$.
\end{proof}

\begin{lemma}[Multiscale contraction lemma]
\label{lem:contraction}
Let $g \in H^\tau(\Omega)$, $\delta \in (0,1]$, $\mu \in (0,1)$, $0 < \delta' \leq \mu \delta$ and let $Z \subset \Omega$ be finite. Suppose that
\begin{align}
\label{eq:contraction_hypothesis}
\Vert g - I_{\delta,Z} g \Vert_{L^2(\Omega)} \leq \beta \Vert E g \Vert_{\Phi_\delta}
\end{align}
for some $\beta \geq 0$.
Then
\begin{align}
\label{eq:contraction_conclusion}
\Vert E (g - I_{\delta,Z} g) \Vert_{\Phi_{\delta'}} \leq \widetilde{C} \left( \beta + \mu^\tau \right) \Vert E g \Vert_{\Phi_\delta},
\qquad
\widetilde{C}  \isdef \left( \frac{2^\tau}{c_\Phi} \right)^{1/2} \max\left\{ C_0,\ C_\tau C_\Phi^{1/2} \right\},
\end{align}
where $C_0$ and $C_\tau$ are the constants of \Cref{prop:universal_ext_op} for $\sigma = 0$ and $\sigma = \tau$.
\end{lemma}

\begin{proof}
Set $r \isdef g - I_{\delta,Z} g$.
By \eqref{eq:fourier_decay_scaled} and $(1+x)^\tau \leq 2^\tau (1 + x^\tau)$ for $x \geq 0$, we have that
\begin{align*}
\Vert E r \Vert^2_{\Phi_{\delta'}}
\leq \frac{(2\pi)^{-d/2}}{c_\Phi} \int_{\R^d} | \widehat{E r}(\omega) |^2 \left( 1 + \delta'^2 \Vert \omega \Vert_2^2 \right)^\tau \mathrm{d}\omega
\leq \frac{2^\tau}{c_\Phi} \left( I_1 + I_2 \right),
\end{align*}
with
$I_1 \isdef (2\pi)^{-d/2} \int_{\R^d} | \widehat{E r}(\omega) |^2 \mathrm{d}\omega$ and
$I_2 \isdef (2\pi)^{-d/2} \delta'^{2\tau} \int_{\R^d} | \widehat{E r}(\omega) |^2 \Vert \omega \Vert_2^{2\tau} \mathrm{d}\omega$.
By Plancherel's theorem, \Cref{prop:universal_ext_op} with $\sigma = 0$ and the hypothesis \eqref{eq:contraction_hypothesis},
\begin{align*}
I_1 \leq \Vert E r \Vert_{L^2(\R^d)}^2 \leq C_0^2 \Vert r \Vert_{L^2(\Omega)}^2 \leq C_0^2 \beta^2 \Vert Eg \Vert_{\Phi_\delta}^2 .
\end{align*}
By \Cref{prop:universal_ext_op} with $\sigma = \tau$, \Cref{lem:interp_error_Htau} and $\delta' \leq \mu \delta$,
\begin{align*}
I_2 \leq \delta'^{2\tau} \Vert E r \Vert_{H^\tau(\R^d)}^2
\leq \delta'^{2\tau} C_\tau^2 \Vert r \Vert_{H^\tau(\Omega)}^2
\leq C_\tau^2 C_\Phi \left( \frac{\delta'}{\delta} \right)^{2\tau} \Vert E g \Vert_{\Phi_\delta}^2
\leq C_\tau^2 C_\Phi \mu^{2\tau} \Vert E g \Vert_{\Phi_\delta}^2 .
\end{align*}
Taking the square root and using $\sqrt{a^2 + b^2} \leq a + b$ gives \eqref{eq:contraction_conclusion}.
\end{proof}

The lemma says that a level contracts as soon as its residual is small in $L^2(\Omega)$ relative to the RKHS norm of its input. The factor $\mu^\tau$ is the price of the change of scale and is independent of how the level was built.
Iterating the lemma over the levels of \Cref{alg:standard_multiscale_algorithm} or \Cref{alg:multiscale_greedy} gives the following consequence, which we use repeatedly.

\begin{cor}
\label{cor:iteration}
Let $f \in H^\tau(\Omega)$, let $Z_1, Z_2, \ldots \subset \Omega$ be finite sets, and let $\delta_1 \in (0,1]$ and $\delta_{\ell+1} \leq \mu \delta_\ell$ for all $\ell$ with a fixed $\mu \in (0,1)$.
Consider the residuals
\begin{align*}
e_0 \isdef f, \qquad e_\ell \isdef e_{\ell-1} - I_{\delta_\ell, Z_\ell} e_{\ell-1} \qquad \text{for } \ell \geq 1,
\end{align*}
so that $f_n = f - e_n$ is the approximation after $n$ levels.
Suppose that every level reduces its input in $L^2(\Omega)$, i.e.\ there are $\beta_\ell \geq 0$ with
\begin{align}
\label{eq:iteration_hypothesis}
\Vert e_\ell \Vert_{L^2(\Omega)} \leq \beta_\ell \Vert E e_{\ell-1} \Vert_{\Phi_{\delta_\ell}} \qquad \text{for } \ell = 1, \dots, n .
\end{align}
Then, with $\alpha_\ell \isdef \widetilde{C} (\beta_\ell + \mu^\tau)$, the residuals contract from one scale to the next,
\begin{align}
\label{eq:iteration_contraction}
\Vert E e_\ell \Vert_{\Phi_{\delta_{\ell+1}}} \leq \alpha_\ell \Vert E e_{\ell-1} \Vert_{\Phi_{\delta_\ell}} \qquad \text{for } \ell = 1, \dots, n-1,
\end{align}
and consequently
\begin{align}
\label{eq:iteration_error}
\Vert f - f_n \Vert_{L^2(\Omega)} \leq c_\Phi^{-1/2} C_\tau \beta_n \prod_{\ell=1}^{n-1} \alpha_\ell \Vert f \Vert_{H^\tau(\Omega)} .
\end{align}
\end{cor}

\begin{proof}
Every $e_\ell$ lies in $H^\tau(\Omega)$, since $f$ does and kernel interpolants lie in $\calh_{k_\delta}(\Omega) = H^\tau(\Omega)$.
Hence \eqref{eq:iteration_contraction} comes from \Cref{lem:contraction} applied with $g = e_{\ell-1}$, $\delta = \delta_\ell$, $\delta' = \delta_{\ell+1}$, $Z = Z_\ell$ and the $\beta$ of \eqref{eq:iteration_hypothesis}.
Iterating \eqref{eq:iteration_contraction} down to $\ell = 1$ and using $\Vert E e_0 \Vert_{\Phi_{\delta_1}} = \Vert E f \Vert_{\Phi_{\delta_1}} \leq c_\Phi^{-1/2} C_\tau \Vert f \Vert_{H^\tau(\Omega)}$ from \eqref{eq:Htau_to_native} gives
\begin{align*}
\Vert E e_{n-1} \Vert_{\Phi_{\delta_n}} \leq c_\Phi^{-1/2} C_\tau \prod_{\ell=1}^{n-1} \alpha_\ell \Vert f \Vert_{H^\tau(\Omega)} ,
\end{align*}
and \eqref{eq:iteration_hypothesis} for $\ell = n$ turns this into \eqref{eq:iteration_error}.
\end{proof}
\subsection{Error bounds for the classical multiscale method}
\label{subsec:error_bounds_standard}

As a first application we recover \Cref{th:conv_result_wendland} and extend it from $L^2(\Omega)$ to $L^p(\Omega)$, $2 \leq p \leq \infty$, in terms of the fill distance of the finest level.
This is a straightforward generalization of \cite[Theorem 1]{wendland2010multiscale} and \cite[Theorem 10]{wendland2017multiscale} and serves as the baseline against which the greedy variants have to be measured.

\begin{theorem}
\label{th:multiscale_using_fill_dist}
Let $X_1, X_2, \dots \subset \Omega$ have fill distances $h_j = h_{X_j, \Omega}$ with
\begin{align}
\label{eq:classical_fill_assumption}
c \mu h_j \leq h_{j+1} \leq \mu h_j \qquad \text{for } j = 1, 2, \dots
\end{align}
for fixed $\mu \in (0,1)$, $c \in (0,1]$, and $h_1 \leq \min\{h_0, \mu\}$, where $h_0$ is the threshold of \Cref{th:sampling_ineq}.
Let $\delta_j  \isdef \nu h_j$ with $1 \leq \nu \leq \mu^{-1}$, so that $\delta_j \leq 1$, let $f \in H^\tau(\Omega)$ and let $f_n$ be the output of \Cref{alg:standard_multiscale_algorithm}.
Then, with $C_2$ the constant of \Cref{th:sampling_ineq},
\begin{align}
\label{eq:classical_contraction}
\Vert E e_j \Vert_{\Phi_{j+1}} \leq \widetilde{C} \left( C_2 C_\Phi^{1/2} \nu^{-\tau} + \mu^\tau \right) \Vert E e_{j-1} \Vert_{\Phi_j} \qquad \text{for } j = 1, 2, \dots .
\end{align}
Moreover, for every $2 \leq p \leq \infty$ and every $\varepsilon \in (0,1)$ there is $\mu_0 \in (0,1)$ such that for $\nu = \mu^{-1}$ and $\mu \leq \mu_0$
\begin{align}
\label{eq:classical_rate}
\Vert f - f_n \Vert_{L^p(\Omega)} \leq C h_n^{(1-\varepsilon)(\tau - d(1/2-1/p))} \Vert f \Vert_{H^\tau(\Omega)} \qquad \text{for all } n \in \N,
\end{align}
with $C$ independent of $n$.
Only the upper bound in \eqref{eq:classical_fill_assumption} is needed for \eqref{eq:classical_contraction}.
\end{theorem}

\begin{proof}
Since $e_j$ vanishes on $X_j$ and $h_j \leq h_0$, \eqref{eq:sampling_Lp} with $p = 2$ and \Cref{lem:interp_error_Htau} give
\begin{align*}
\Vert e_j \Vert_{L^2(\Omega)} \leq C_2 h_j^\tau \Vert e_j \Vert_{H^\tau(\Omega)} \leq C_2 h_j^\tau C_\Phi^{1/2} \delta_j^{-\tau} \Vert E e_{j-1} \Vert_{\Phi_j} = C_2 C_\Phi^{1/2} \nu^{-\tau} \Vert E e_{j-1} \Vert_{\Phi_j} ,
\end{align*}
i.e.\ \eqref{eq:contraction_hypothesis} holds with $\beta = C_2 C_\Phi^{1/2} \nu^{-\tau}$ on every level, and since $\delta_{j+1} = \nu h_{j+1} \leq \mu \delta_j$, \eqref{eq:classical_contraction} is \Cref{lem:contraction}.

For the rate, \eqref{eq:sampling_Lp} and \Cref{lem:interp_error_Htau} give
\begin{align*}
\Vert e_n \Vert_{L^p(\Omega)} \leq C_2 h_n^{\tau - d(1/2 - 1/p)} \Vert e_n \Vert_{H^\tau(\Omega)}
\leq C_2 C_\Phi^{1/2} \nu^{-\tau} h_n^{-d(1/2-1/p)} \Vert E e_{n-1} \Vert_{\Phi_n} .
\end{align*}
With $\nu = \mu^{-1}$ the contraction factor in \eqref{eq:classical_contraction} is $\alpha = C_3 \mu^\tau$ with $C_3  \isdef \widetilde{C}(C_2 C_\Phi^{1/2} + 1)$, so \Cref{cor:iteration} yields
\begin{align*}
\Vert e_n \Vert_{L^p(\Omega)} \leq C_2 C_\Phi^{1/2} c_\Phi^{-1/2} C_\tau \mu^{\tau} h_n^{-d(1/2-1/p)} C_3^{n-1} \mu^{\tau(n-1)} \Vert f \Vert_{H^\tau(\Omega)} .
\end{align*}
The lower bound in \eqref{eq:classical_fill_assumption} gives $h_n \geq (c\mu)^{n-1} h_1$, i.e.\ $\mu^{\tau(n-1)} \leq c^{-\tau(n-1)} (h_n/h_1)^\tau$, and therefore
\begin{align*}
\Vert e_n \Vert_{L^p(\Omega)} \leq C_4 h_n^{\tau - d(1/2-1/p)} \left( C_3 c^{-\tau} \right)^{n-1} \Vert f \Vert_{H^\tau(\Omega)},
\qquad C_4  \isdef C_2 C_\Phi^{1/2} c_\Phi^{-1/2} C_\tau \mu^\tau h_1^{-\tau} .
\end{align*}
It remains to absorb the geometric factor.
Set $\gamma_p  \isdef \tau - d(1/2 - 1/p) \geq \tau - d/2 > 0$ and split $h_n^{\gamma_p} = h_n^{(1-\varepsilon)\gamma_p} h_n^{\varepsilon \gamma_p}$.
The upper bound in \eqref{eq:classical_fill_assumption} gives $h_n \leq \mu^{n-1} h_1$, hence
\begin{align*}
h_n^{\varepsilon \gamma_p} \left( C_3 c^{-\tau} \right)^{n-1} \leq h_1^{\varepsilon \gamma_p} \left( \mu^{\varepsilon \gamma_p} C_3 c^{-\tau} \right)^{n-1} \leq h_1^{\varepsilon \gamma_p}
\end{align*}
as soon as $\mu^{\varepsilon (\tau - d/2)} C_3 c^{-\tau} \leq 1$, i.e.\ for $\mu \leq \mu_0  \isdef (C_3 c^{-\tau})^{-1/(\varepsilon(\tau - d/2))}$, which is independent of $p$.
This proves \eqref{eq:classical_rate} with $C  \isdef C_4 h_1^{\varepsilon \gamma_p}$.
\end{proof}

For $p = 2$ the rate is $h_n^{(1-\varepsilon)\tau}$ and for $p = \infty$ it is $h_n^{(1-\varepsilon)(\tau - d/2)}$.
For quasi-uniform sets with $|X_n| \asymp h_n^{-d}$ these are the rates $|X_n|^{-(1-\varepsilon)\tau/d}$ and $|X_n|^{-(1-\varepsilon)(\tau/d - 1/2)}$ of single-scale interpolation, up to the loss $\varepsilon$: as anticipated, the multiscale construction does not change the approximation order of the classical method.
\subsection{\texorpdfstring{Error bounds for multiscale $f$-greedy}{Error bounds for multiscale f-greedy}}
\label{subsec:error_bounds_greedy_multiscale}

For adaptively distributed centres the fill distance is not the right quantity: there is no need to fill the whole domain.
We therefore state the convergence of the multiscale $f$-greedy algorithm in terms of the number of centres.
The per-level estimate is the greedy bound of \Cref{cor:f_greedy_rescaled_kernel}, and the mechanism is again \Cref{lem:contraction}.
We write $c_{\mathrm{g}}$ for the constant of \Cref{cor:f_greedy_rescaled_kernel}.

\begin{theorem}
\label{th:main_theorem}
Let $\Omega$ have a sufficiently smooth boundary in the sense of \Cref{prop:geom_mean_powerfunc_rescaled} and let $\Phi$ satisfy \ref{ass:integrable} and \ref{ass:fourier}.
Let $\mu \in (0,1)$, $c \in (0,1]$, and let $M_1, M_2, \ldots \in \N$ satisfy
\begin{align}
\label{eq:assumption_Mj}
\mu^{-d} M_\ell \leq M_{\ell+1} \leq (c\mu)^{-d} M_\ell \qquad \text{for } \ell = 1, 2, \dots .
\end{align}
Let $1 \leq \nu \leq M_1^{1/d}$ and $\delta_\ell  \isdef \nu M_\ell^{-1/d}$, so that $\delta_\ell \in (0,1]$ and $\delta_{\ell+1} \leq \mu \delta_\ell$.
Let $f \in H^\tau(\Omega)$ and let $f_n$ be the output of \Cref{alg:multiscale_greedy} with the $f$-greedy rule, $\beta = 1$ in \eqref{eq:beta_greedy}, where line~\ref{line:best_iterate} minimizes over $0 \leq i \leq M_\ell$ with $X_0 = \emptyset$ and $I_{\delta,\emptyset}  \isdef 0$.
Then
\begin{align}
\label{eq:theorem_first_eq}
\Vert E e_{\ell} \Vert_{\Phi_{\ell+1}} \leq \widetilde{C} \left( |\Omega|^{1/2} c_{\mathrm{g}} \nu^{-\tau} + \mu^\tau \right) \Vert E e_{\ell-1} \Vert_{\Phi_{\ell}} \qquad \text{for } \ell = 1, 2, \dots .
\end{align}
Moreover, for every $\varepsilon \in (0,1)$ there is $\mu_0 \in (0,1)$ such that for $\nu = \mu^{-1}$, which is admissible if $M_1 \geq \mu^{-d}$, and $\mu \leq \mu_0$,
\begin{align}
\label{eq:main_rate}
\Vert f - f_n \Vert_{L^\infty(\Omega)} \leq C M_n^{-(1-\varepsilon)\tau/d} \Vert f \Vert_{H^\tau(\Omega)} \qquad \text{for all } n \in \N,
\end{align}
with $C$ independent of $n$.
\end{theorem}

\begin{proof}
Fix a level $\ell$.
By \eqref{eq:Htau_to_native}, $e_{\ell-1} \in H^\tau(\Omega) \subset \calh_{k_{\delta_\ell}}(\Omega)$, 
so \Cref{cor:f_greedy_rescaled_kernel} applies to the $f$-greedy run of the level with $n = M_\ell + 1$ steps.
Since line~\ref{line:best_iterate} selects the iterate with the smallest residual among $X_0, \dots, X_{M_\ell}$,
\begin{align}
\label{eq:proof_greedy_bound}
\Vert e_\ell \Vert_{L^\infty(\Omega)}
&= \min_{0 \leq i \leq M_\ell} \Vert e_{\ell-1} - I_{\delta_\ell, X_i} e_{\ell-1} \Vert_{L^\infty(\Omega)} \notag \\
&\leq c_{\mathrm{g}} \delta_\ell^{-\tau} (M_\ell + 1)^{-\tau/d} \Vert e_{\ell-1} \Vert_{\calh_{k_{\delta_\ell}}(\Omega)}
\leq c_{\mathrm{g}} \nu^{-\tau} \Vert E e_{\ell-1} \Vert_{\Phi_\ell},
\end{align}
where we used $\delta_\ell^{-\tau} M_\ell^{-\tau/d} = \nu^{-\tau}$ and \Cref{lem:native_space_omega}(i).
Hence $\Vert e_\ell \Vert_{L^2(\Omega)} \leq |\Omega|^{1/2} \Vert e_\ell \Vert_{L^\infty(\Omega)} \leq |\Omega|^{1/2} c_{\mathrm{g}} \nu^{-\tau} \Vert E e_{\ell-1} \Vert_{\Phi_\ell}$, which is \eqref{eq:contraction_hypothesis} with $\beta = |\Omega|^{1/2} c_{\mathrm{g}} \nu^{-\tau}$.
Since $\delta_{\ell+1}/\delta_\ell = (M_\ell / M_{\ell+1})^{1/d} \leq \mu$ by \eqref{eq:assumption_Mj}, \Cref{lem:contraction} gives \eqref{eq:theorem_first_eq}.

For the rate let $\nu = \mu^{-1}$, so that the contraction factor is $\alpha = C_5 \mu^\tau$ with $C_5  \isdef \widetilde{C}(|\Omega|^{1/2} c_{\mathrm{g}} + 1)$.
By \eqref{eq:proof_greedy_bound} on level $n$ and \Cref{cor:iteration},
\begin{align*}
\Vert f - f_n \Vert_{L^\infty(\Omega)} = \Vert e_n \Vert_{L^\infty(\Omega)}
\leq c_{\mathrm{g}} \mu^{\tau} \Vert E e_{n-1} \Vert_{\Phi_n}
\leq c_{\mathrm{g}} \mu^{\tau} c_\Phi^{-1/2} C_\tau C_5^{n-1} \mu^{\tau(n-1)} \Vert f \Vert_{H^\tau(\Omega)} .
\end{align*}
The upper bound in \eqref{eq:assumption_Mj} gives $M_n \leq (c\mu)^{-d(n-1)} M_1$, i.e.\ $\mu^{\tau(n-1)} \leq c^{-\tau(n-1)} (M_1/M_n)^{\tau/d}$, and the lower bound gives $M_n \geq \mu^{-d(n-1)} M_1$, i.e.\ $M_n^{-\varepsilon\tau/d} \leq \mu^{\varepsilon\tau(n-1)} M_1^{-\varepsilon\tau/d}$.
Splitting $M_n^{-\tau/d} = M_n^{-(1-\varepsilon)\tau/d} M_n^{-\varepsilon\tau/d}$ we obtain
\begin{align*}
\Vert f - f_n \Vert_{L^\infty(\Omega)}
\leq c_{\mathrm{g}} \mu^{\tau} c_\Phi^{-1/2} C_\tau M_1^{(1-\varepsilon)\tau/d} M_n^{-(1-\varepsilon)\tau/d} \left( \mu^{\varepsilon \tau} C_5 c^{-\tau} \right)^{n-1} \Vert f \Vert_{H^\tau(\Omega)} ,
\end{align*}
and the last factor is at most one for $\mu \leq \mu_0  \isdef (C_5 c^{-\tau})^{-1/(\varepsilon\tau)}$.
This is \eqref{eq:main_rate} with $C  \isdef c_{\mathrm{g}} \mu^{\tau} c_\Phi^{-1/2} C_\tau M_1^{(1-\varepsilon)\tau/d}$.
\end{proof}

\begin{cor}
\label{cor:total_centres}
Under the assumptions of \Cref{th:main_theorem}, the total number of centres $N_n  \isdef \sum_{\ell=1}^n |Z_\ell|$ satisfies $N_n \leq M_n / (1 - \mu^d)$, and hence
\begin{align*}
\Vert f - f_n \Vert_{L^\infty(\Omega)} \leq C (1-\mu^d)^{-(1-\varepsilon)\tau/d} N_n^{-(1-\varepsilon)\tau/d} \Vert f \Vert_{H^\tau(\Omega)} .
\end{align*}
\end{cor}

\begin{proof}
By \eqref{eq:assumption_Mj}, $M_\ell \leq \mu^{d(n-\ell)} M_n$, so $N_n \leq \sum_{\ell=1}^n M_\ell \leq M_n \sum_{k \geq 0} \mu^{dk}$, and $M_n \geq (1-\mu^d) N_n$ in \eqref{eq:main_rate}.
\end{proof}

{By \eqref{eq:beta_greedy_error} with $\beta = 1$, the single-scale $f$-greedy algorithm converges in $L^\infty(\Omega)$ with the order $\tau/d$ in the number of centres, and by \Cref{cor:total_centres} the multiscale $f$-greedy algorithm converges with the same rate, up to the arbitrarily small loss $\varepsilon$.
Compared with the classical multiscale method of \Cref{th:multiscale_using_fill_dist} for $p = \infty$, this is a gain of the factor $N_n^{1/2}$, which is the known gain of target-dependent selection over quasi-uniform centres \cite{wenzel2023analysis}.
In $L^2(\Omega)$ there is no gain, since \Cref{th:multiscale_using_fill_dist} with $p = 2$ gives the same order $\tau/d$ for quasi-uniform levels.}

\subsection{Stability of the level matrices and the trade-off with convergence}
\label{subsec:stability}

The role of the multiscale structure becomes visible in the conditioning of the level matrices, for which we prove a two-sided estimate.
The upper bound controls the condition number of a kernel matrix with rescaled kernel by one dimensionless quantity, the ratio $\delta / q_X$ between the scale of the kernel and the separation distance of the centres.
It combines the lower bound on the smallest eigenvalue of \cite{narcowich1991norms}, in the form of \cite[Theorem~12.3]{wendland2005scattered}, with Gershgorin's theorem and the scaling \eqref{eq:rescaled_kernel}.
The lower bound shows that at a fixed scale the condition number grows with the number of centres, for any distribution of the centres.
It follows from the entropy number estimates behind the greedy convergence rates through the uncertainty relation of \cite{schaback1995error}.
Sharper stability estimates for specific kernels are available in \cite{diederichs2019improved}, and the relation between conditioning and the stability of the interpolation process is discussed in \cite{demarchi2010stability}.

\begin{prop}
\label{prop:stability}
Let $\Phi$ satisfy \ref{ass:integrable} and \ref{ass:fourier}, and assume in addition that $\Phi$ decays fast enough at infinity such that
\begin{align}
\label{eq:decay_space}
D_\Phi \isdef \sum_{j=0}^\infty (j+1)^d \sup_{\Vert x \Vert_2 \geq j} |\Phi(x)| < \infty ,
\end{align}
which holds if $|\Phi(x)| \leq C (1+\Vert x \Vert_2)^{-(d+1+\epsilon)}$ for some $\epsilon > 0$, in particular for the Mat\'ern kernels and for compactly supported kernels.
Let $\delta \in (0,1]$ and let $X \subset \R^d$ be finite with separation distance $q_X > 0$.
Then there are constants $c_1, c_2 > 0$, depending only on $\Phi$ and $d$, such that
\begin{align*}
\lambda_{\min}(\Kmat{\delta}{X}) \geq c_1 \delta^{-d} 
\min\bigg\{ 1, \frac{q_X}{\delta} \bigg\}^{2\tau - d}
\qquad \text{and} \qquad
\lambda_{\max}(\Kmat{\delta}{X}) \leq c_2 \delta^{-d} 
\max\bigg\{ 1, \frac{\delta}{q_X} \bigg\}^{d},
\end{align*}
and consequently
\begin{align}
\label{eq:stability}
\kappa_2(\Kmat{\delta}{X}) \leq C_\kappa \max\bigg\{ 1, \frac{\delta}{q_X} \bigg\}^{2\tau}, \qquad C_\kappa \isdef \frac{c_2}{c_1} .
\end{align}
\end{prop}

\begin{proof}
Let $Y \isdef \delta^{-1} X$.
By \eqref{eq:rescaled_kernel}, $\Kmat{\delta}{X} = \delta^{-d} \Kmat{1}{Y}$ and $q_Y = q_X / \delta$, so it suffices to show
\begin{align}
\label{eq:proof_stability_goal}
\lambda_{\min}(\Kmat{1}{Y}) \geq c_1 \min\{1, q_Y\}^{2\tau - d}
\qquad \text{and} \qquad
\lambda_{\max}(\Kmat{1}{Y}) \leq c_2 \max\{1, q_Y^{-1}\}^{d} .
\end{align}

\noindent \emph{Lower bound.}
By \ref{ass:integrable} and \ref{ass:fourier}, $\Phi \in L^1(\R^d) \cap C(\R^d)$ is even with positive Fourier transform, so \cite[Theorem~12.3]{wendland2005scattered} applies.
There are constants $C_d, M_d > 0$, depending only on $d$, such that
\begin{align*}
\lambda_{\min}(\Kmat{1}{Y}) \geq C_d \varphi_0(M) M^d
\end{align*}
for every $M \geq M_d / q_Y$ and with $\varphi_0(M) \isdef \inf_{\Vert \omega \Vert_2 \leq 2M} \widehat{\Phi}(\omega)$.
We choose $M \isdef \max\{M_d, 1\} / \min\{q_Y, 1\}$, which satisfies $M \geq M_d / q_Y$ and $M \geq 1$.
By \eqref{eq:fourier_decay}, $\varphi_0(M) \geq c_\Phi (1 + 4M^2)^{-\tau} \geq c_\Phi 5^{-\tau} M^{-2\tau}$, and since $d - 2\tau < 0$,
\begin{align*}
\lambda_{\min}(\Kmat{1}{Y}) \geq C_d c_\Phi 5^{-\tau} M^{d - 2\tau}
= C_d c_\Phi 5^{-\tau} \max\{M_d, 1\}^{d - 2\tau} \min\{q_Y, 1\}^{2\tau - d} ,
\end{align*}
which is the first bound in \eqref{eq:proof_stability_goal} with $c_1 \isdef C_d c_\Phi 5^{-\tau} \max\{M_d, 1\}^{d - 2\tau}$.

\noindent \emph{Upper bound.}
By Gershgorin's theorem, it holds that
$$
\lambda_{\max}(\Kmat{1}{Y}) \leq \max_{y_i \in Y} \sum_{y \in Y} |\Phi(y_i - y)|.
$$
Fix $y_i$ and let $n_j$ be the number of $y \in Y$ with $j \leq \Vert y - y_i \Vert_2 < j+1$, $j \in \N_0$.
The balls $B(y, q_Y)$, $y \in Y$, are pairwise disjoint, and those with $\Vert y - y_i \Vert_2 < j+1$ lie in $B(y_i, j+1+q_Y)$, so a comparison of volumes gives
\begin{align*}
n_j \leq \Big( 1 + \frac{j+1}{q_Y} \Big)^d \leq (j+2)^d \max\{1, q_Y^{-1}\}^d \leq 2^d (j+1)^d \max\{1, q_Y^{-1}\}^d .
\end{align*}
Hence, by \eqref{eq:decay_space},
\begin{align*}
\sum_{y \in Y} |\Phi(y_i - y)| \leq \sum_{j=0}^\infty n_j \sup_{\Vert x \Vert_2 \geq j} |\Phi(x)|
\leq 2^d D_\Phi \max\{1, q_Y^{-1}\}^d ,
\end{align*}
which is the second bound in \eqref{eq:proof_stability_goal} with $c_2 \isdef 2^d D_\Phi$.
Finally, $\min\{q_Y, 1\}^{-(2\tau - d)} = \max\{1, q_Y^{-1}\}^{2\tau - d}$, so the two bounds give
$$
\kappa_2(\Kmat{1}{Y}) \leq (c_2 / c_1) \max\{1, q_Y^{-1}\}^{2\tau},
$$ which is \eqref{eq:stability} with $q_Y = q_X / \delta$.
\end{proof}

The upper bound says that a level matrix is well conditioned as soon as its centres are separated at the scale of its kernel.
It leaves open whether a clever choice of centres could keep the condition number small at a fixed scale.
It cannot.
The entropy number estimates that give the greedy convergence rates also bound the condition number from below, for every set of centres, and the link is the uncertainty relation of Schaback \cite{schaback1995error}, see also \cite[Chapter~12]{wendland2005scattered}.

\begin{lemma}
\label{lem:uncertainty}
Let $X \subset \R^d$ be finite, $x \notin X$ and $\delta \in (0,1]$.
Then the power function of $X$ with respect to $k_\delta$ satisfies $P_{\delta,X}(x)^2 \geq \lambda_{\min}(\Kmat{\delta}{X \cup \{x\}})$.
\end{lemma}

\begin{proof}
Let $A \isdef \Kmat{\delta}{X \cup \{x\}}$ with $x$ as the last point.
The Schur complement of $\Kmat{\delta}{X}$ in $A$ is
\begin{align*}
k_\delta(x,x) - k_\delta(x,X) \Kmat{\delta}{X}^{-1} k_\delta(X,x) = P_{\delta,X}(x)^2 ,
\end{align*}
and its reciprocal is the last diagonal entry of $A^{-1}$, which is at most $\lambda_{\max}(A^{-1}) = 1 / \lambda_{\min}(A)$.
\end{proof}

\begin{theorem}
\label{th:tradeoff}
Let $\Omega$ have a sufficiently smooth boundary in the sense of \Cref{prop:geom_mean_powerfunc_rescaled}, let $\Phi$ satisfy \ref{ass:integrable} and \ref{ass:fourier}, and let $\delta \in (0,1]$.
Then every set $X_n \subset \Omega$ of $n$ distinct points satisfies
\begin{align}
\label{eq:lambda_min_upper}
\lambda_{\min}(\Kmat{\delta}{X_n}) \leq c_{\mathrm{g}}^2 \delta^{-2\tau} n^{1 - 2\tau/d}
\qquad \text{and} \qquad
\kappa_2(\Kmat{\delta}{X_n}) \geq c_0 (\delta n^{1/d})^{2\tau},
\end{align}
where $c_{\mathrm{g}}$ is the constant of \Cref{prop:geom_mean_powerfunc_rescaled} and $c_0 > 0$ depends only on $\Phi$, $\Omega$ and $d$.
\end{theorem}

\begin{proof}
Order the points of $X_n$ arbitrarily as $x_1, \dots, x_n$ and let $X_i \isdef \{x_1, \dots, x_i\}$.
\Cref{prop:geom_mean_powerfunc_rescaled} holds for every sequence of points, and the minimum of $n$ numbers is at most their geometric mean, so there is an index $0 \leq i^* \leq n-1$ with
\begin{align*}
P_{\delta, i^*}(x_{i^*+1}) \leq c_{\mathrm{g}} \delta^{-\tau} n^{-\tau/d + 1/2} .
\end{align*}
By \Cref{lem:uncertainty}, $\lambda_{\min}(\Kmat{\delta}{X_{i^*+1}}) \leq P_{\delta,i^*}(x_{i^*+1})^2$, and since $\Kmat{\delta}{X_{i^*+1}}$ is a principal submatrix of $\Kmat{\delta}{X_n}$, Cauchy's interlacing theorem gives $\lambda_{\min}(\Kmat{\delta}{X_n}) \leq \lambda_{\min}(\Kmat{\delta}{X_{i^*+1}})$.
This proves the first inequality.

For the second we need a lower bound on $\lambda_{\max}(\Kmat{\delta}{X_n})$ that grows with $n$.
Since $\Phi$ is continuous and $\Phi(0) > 0$ by \ref{ass:fourier}, there is $r_\Phi > 0$ with $\Phi \geq \Phi(0)/2$ on $\overline{B(0, r_\Phi)}$.
We cover $\Omega$ by at most $C_{\mathrm{c}} \delta^{-d}$ balls of radius $r_\Phi \delta / 2$, with $C_{\mathrm{c}} \isdef (2 \sqrt{d} \diam \Omega / r_\Phi + 1)^d$.
One of them contains $m \geq n \delta^d / C_{\mathrm{c}}$ points of $X_n$, any two of which are at distance at most $r_\Phi \delta$, so that every entry of $\Kmat{\delta}{X_n}$ between them is at least $\delta^{-d} \Phi(0) / 2$.
For the indicator vector $v \in \{0,1\}^n$ of these points,
\begin{align*}
\lambda_{\max}(\Kmat{\delta}{X_n}) \geq \frac{v^\top \Kmat{\delta}{X_n} v}{v^\top v} \geq \frac{m^2 \delta^{-d} \Phi(0) / 2}{m} \geq \frac{\Phi(0)}{2 C_{\mathrm{c}}} n .
\end{align*}
Dividing by the first inequality gives the second with $c_0 \isdef \Phi(0) / (2 C_{\mathrm{c}} c_{\mathrm{g}}^2)$.
\end{proof}
Together, \Cref{prop:stability} and \Cref{th:tradeoff} read
\begin{align}
\label{eq:two_sided}
c_0 \chi^{2\tau} \leq \kappa_2(\Kmat{\delta}{X_n}) \leq C_\kappa \max\bigg\{1, \frac{\delta}{q_{X_n}}\bigg\}^{2\tau},
\qquad \chi \isdef \delta n^{1/d} .
\end{align}
The lower bound depends on the centres only through their number, the upper bound only through their separation, and the two agree up to constants for quasi-uniform sets, where $q_{X_n} \asymp n^{-1/d}$ and hence $\delta / q_{X_n} \asymp \chi$.

The same number governs the accuracy of a level, since the bound of \Cref{cor:f_greedy_rescaled_kernel} is $c_{\mathrm{g}} \delta^{-\tau} n^{-\tau/d} = c_{\mathrm{g}} \chi^{-\tau}$.
For the $f$-greedy algorithm with $n$ centres at the scale $\delta$ we therefore have
\begin{align}
\label{eq:tradeoff}
\min_{0 \leq i \leq n-1} \Vert f - I_{\delta, X_i} f \Vert_{L^\infty(\Omega)} \leq c_{\mathrm{g}} \chi^{-\tau} \Vert f \Vert_{\calh_{k_\delta}(\Omega)}
\qquad \text{and} \qquad
\kappa_2(\Kmat{\delta}{X_n}) \geq c_0 \chi^{2\tau} .
\end{align}
Eliminating $\chi$ gives $c_{\mathrm{g}} \chi^{-\tau} \geq c_{\mathrm{g}} ( c_0 / \kappa_2(\Kmat{\delta}{X_n}) )^{1/2}$, meaning that a guaranteed error $\epsilon$ requires $\kappa_2(\Kmat{\delta}{X_n}) \geq c_0 (c_{\mathrm g} / \epsilon)^2$.
The guaranteed error reduction of a kernel matrix with condition number $\kappa$ is never better than the factor $c_{\mathrm{g}} (c_0 / \kappa)^{1/2}$, whatever $\tau$, $d$ and the position of the centres: every digit of guaranteed accuracy costs two digits of conditioning.
This concerns the guarantee, not the error itself, and it is the trade-off that the multiscale construction resolves, see \Cref{rem:stability_convergence_cases}.

\subsection{A general stability-convergence principle}
\label{subsec:stability_convergence}

We return to the setting of \Cref{cor:iteration}: $f \in H^\tau(\Omega)$, finite sets $Z_\ell \subset \Omega$ with $M_\ell \isdef |Z_\ell|$, scales $\delta_1 \in (0,1]$ and $\delta_{\ell+1} \leq \mu \delta_\ell$ for a fixed $\mu \in (0,1)$, residuals $e_0 = f$, $e_\ell = e_{\ell-1} - I_{\delta_\ell, Z_\ell} e_{\ell-1}$, and approximants $f_n = f - e_n$.
Nothing is assumed on how the $Z_\ell$ are produced.
A level enters through its reductions $\beta_\ell, \gamma_\ell \geq 0$ in $L^2(\Omega)$ and in $L^\infty(\Omega)$, as follows
\begin{align}
\label{eq:level_reduction}
\Vert e_\ell \Vert_{L^2(\Omega)} \leq \beta_\ell \Vert E e_{\ell-1} \Vert_{\Phi_{\delta_\ell}}
\qquad \text{and} \qquad
\Vert e_\ell \Vert_{L^\infty(\Omega)} \leq \gamma_\ell \Vert E e_{\ell-1} \Vert_{\Phi_{\delta_\ell}} .
\end{align}
The first is the hypothesis of \Cref{lem:contraction} and drives the contraction across the levels.
In \Cref{th:multiscale_using_fill_dist} it came from the sampling inequality and in \Cref{th:main_theorem} from the greedy bound.
The second is only needed on the last level, to measure the final error in $L^\infty(\Omega)$.
The following theorem collects what the contraction lemma and the stability estimates say when the level constants are uniform.

\begin{theorem}
\label{th:stability_convergence}
Let $\Phi$ satisfy \ref{ass:integrable}--\ref{ass:fourier} and let $\Omega$ have a sufficiently smooth boundary in the sense of \Cref{prop:geom_mean_powerfunc_rescaled}.
\begin{enumerate}[label=(\roman*)]
\item \emph{(convergence)}
If $\beta_\ell \leq \bar\beta$ for all $\ell$, then with $\alpha \isdef \widetilde{C} (\bar\beta + \mu^\tau)$ it holds
\begin{align}
\label{eq:general_convergence}
\Vert f - f_n \Vert_{L^2(\Omega)} \leq c_\Phi^{-1/2} C_\tau \bar\beta \alpha^{n-1} \Vert f \Vert_{H^\tau(\Omega)}
\end{align}
for all $n \in \N$.
The same bound holds in $L^\infty(\Omega)$ with $\bar\beta$ replaced by $\gamma_n$, and $f_n \to f$ linearly in $L^2(\Omega)$ if $\alpha < 1$.
\item \emph{(stability)}
Suppose also \eqref{eq:decay_space} holds. Then, every level matrix satisfies
\begin{align}
\label{eq:general_conditioning}
c_0 \chi_\ell^{2\tau} \leq \kappa_2(\Kmat{\delta_\ell}{Z_\ell}) \leq C_\kappa \max\bigg\{ 1, \frac{\delta_\ell}{q_{Z_\ell}} \bigg\}^{2\tau},
\qquad \chi_\ell \isdef \delta_\ell M_\ell^{1/d} .
\end{align}
In particular, if $q_{Z_\ell} \geq c_q \delta_\ell$ for all $\ell$ with a fixed $c_q > 0$, then $\kappa_2(\Kmat{\delta_\ell}{Z_\ell}) \leq C_\kappa \max\{1, c_q^{-1}\}^{2\tau}$ uniformly in $\ell$.
\item \emph{(rate)}
If $\beta_\ell \leq C_\beta \mu^\tau$ for all $\ell$ and $\mu^{-d} M_\ell \leq M_{\ell+1} \leq (c\mu)^{-d} M_\ell$ with $c \in (0,1]$, then for every $\varepsilon \in (0,1)$ there is $\mu_0 \in (0,1)$ such that for $\mu \leq \mu_0$
\begin{align}
\label{eq:general_rate}
\Vert f - f_n \Vert_{L^2(\Omega)} \leq C_\beta C M_n^{-(1-\varepsilon)\tau/d} \Vert f \Vert_{H^\tau(\Omega)}
\end{align}
 for all  $n \in \N$ and with $C \isdef c_\Phi^{-1/2} C_\tau \mu^\tau M_1^{(1-\varepsilon)\tau/d}$. 
Moreover, the same bound holds in $L^\infty(\Omega)$ with $C_\beta$ replaced by $C_\gamma$ if $\gamma_n \leq C_\gamma \mu^\tau$ for all $n$.
\end{enumerate}
\end{theorem}

\begin{proof}
\emph{(i).}
The result comes from \Cref{cor:iteration} with $\alpha_\ell = \widetilde{C}(\beta_\ell + \mu^\tau) \leq \alpha$, and thus $\prod_{\ell=1}^{n-1} \alpha_\ell \leq \alpha^{n-1} $.

\emph{(ii).}
Every $\delta_\ell \leq \mu^{\ell-1} \delta_1 \leq 1$, so \Cref{th:tradeoff} applies to $X_n = Z_\ell$ with $n = M_\ell$ and $\delta = \delta_\ell$ and gives the lower bound $c_0 (\delta_\ell M_\ell^{1/d})^{2\tau} = c_0 \chi_\ell^{2\tau}$, and \Cref{prop:stability} applies to $X = Z_\ell$ and $\delta = \delta_\ell$ and gives the upper bound.
If $q_{Z_\ell} \geq c_q \delta_\ell$, then $\delta_\ell / q_{Z_\ell} \leq c_q^{-1}$ and the upper bound is at most $C_\kappa \max\{1, c_q^{-1}\}^{2\tau}$.

\emph{(iii).}
We apply (i) with $\bar\beta = C_\beta \mu^\tau$.
Then $\alpha = \widetilde{C}(C_\beta + 1) \mu^\tau = C_* \mu^\tau$ with $C_* \isdef \widetilde{C}(C_\beta + 1)$, and \eqref{eq:general_convergence} reads
\begin{align}
\label{eq:proof_general_start}
\Vert f - f_n \Vert_{L^2(\Omega)} \leq c_\Phi^{-1/2} C_\tau C_\beta \mu^\tau C_*^{n-1} \mu^{\tau(n-1)} \Vert f \Vert_{H^\tau(\Omega)} .
\end{align}
Iterating $M_{\ell+1} \leq (c\mu)^{-d} M_\ell$ gives $M_n \leq (c\mu)^{-d(n-1)} M_1$, that is $\mu^{n-1} \leq c^{-(n-1)} (M_1 / M_n)^{1/d}$.
We use this on the $(1-\varepsilon)$-part of $\mu^{\tau(n-1)}$ only and keep the $\varepsilon$-part to absorb $C_*^{n-1}$,
\begin{align*}
C_*^{n-1} \mu^{\tau(n-1)}
= \big( C_* \mu^{\varepsilon\tau} \big)^{n-1} \mu^{(1-\varepsilon)\tau(n-1)}
\leq \big( C_* c^{-(1-\varepsilon)\tau} \mu^{\varepsilon\tau} \big)^{n-1} M_1^{(1-\varepsilon)\tau/d} M_n^{-(1-\varepsilon)\tau/d} .
\end{align*}
For $\mu \leq \mu_0 \isdef \min\{1/2, (C_* c^{-(1-\varepsilon)\tau})^{-1/(\varepsilon\tau)}\}$ the first factor is at most one, and \eqref{eq:proof_general_start} becomes \eqref{eq:general_rate}.
The $L^\infty(\Omega)$ version of (i) replaces the factor $\bar\beta = C_\beta \mu^\tau$ in \eqref{eq:proof_general_start} by $\gamma_n \leq C_\gamma \mu^\tau$ and leaves $\alpha$, hence $C_*$ and $\mu_0$, unchanged, so the same computation gives \eqref{eq:general_rate} with $C_\beta$ replaced by $C_\gamma$.
\end{proof}

\begin{rem}
\label{rem:stability_convergence_cases}
\Cref{th:stability_convergence} reads a level through three quantities: $\beta_\ell$ drives the contraction, $\chi_\ell = \delta_\ell M_\ell^{1/d}$ the lower bound on the condition number, and $\delta_\ell / q_{Z_\ell}$ the upper bound. We can analyze different scenarios based on these quantities.
\begin{enumerate}[label=(\alph*)]
\item \emph{Single scale.}
The scale is fixed and centres are added, so $\chi = \delta n^{1/d} \to \infty$: the condition number grows at least like $c_0 \delta^{2\tau} n^{2\tau/d}$, whatever the centres, and the $f$-greedy bound $c_{\mathrm g} \chi^{-\tau}$ decays for the same reason, see \eqref{eq:tradeoff}.
\item \emph{Classical multiscale.}
For quasi-uniform levels and $\delta_\ell = \nu h_{Z_\ell,\Omega}$, both $\chi_\ell$ and $\delta_\ell / q_{Z_\ell}$ are bounded by multiples of $\nu$, so both sides of \eqref{eq:general_conditioning} are constants.
The sampling inequality gives $\beta_\ell \leq C_2 C_\Phi^{1/2} \nu^{-\tau}$ and $\gamma_n \leq C_2 C_\Phi^{1/2} \nu^{-\tau} h_n^{-d/2}$, and \Cref{th:stability_convergence}(iii) recovers the orders $\tau/d$ in $L^2(\Omega)$ and $\tau/d - 1/2$ in $L^\infty(\Omega)$ of \Cref{th:multiscale_using_fill_dist}.
\item \emph{Multiscale $f$-greedy.}
With $\delta_\ell = \nu M_\ell^{-1/d}$ we have $\chi_\ell = \nu$ on every level, so the lower bound in \eqref{eq:general_conditioning} is the constant $c_0 \nu^{2\tau}$.
\Cref{cor:f_greedy_rescaled_kernel} gives $\gamma_\ell \leq c_{\mathrm g} \nu^{-\tau}$ and $\beta_\ell \leq |\Omega|^{1/2} c_{\mathrm g} \nu^{-\tau}$: the contraction is that of (b), and in $L^\infty(\Omega)$ the last level carries no factor $M_n^{1/2}$, which is the gain of \Cref{th:main_theorem}.
The upper bound is not controlled, since the $f$-greedy rule may cluster centres and $\kappa_2(\Kmat{\delta_\ell}{Z_\ell})$ grows with $\delta_\ell / q_{Z_\ell}$. By \Cref{th:stability_convergence}(ii), rejecting candidates closer than $2 c_q \delta_\ell$ to the current centres prevents this, which is the conditioning-controlled variant of \Cref{subsec:experimental_setting}.
\end{enumerate}
\end{rem}

\section{Numerical results}
\label{sec:numerical_results}

This section is divided into two main parts.
The first one tests the multiscale $f$-greedy algorithm, that is \Cref{alg:multiscale_greedy} with the $f$-greedy rule, on a composite function on the unit square and on the Stanford Bunny, with particular attention to the distribution of the selected centres across the levels and to the conditioning of the level systems.
The second one compares it with the single-scale greedy algorithms, the multiscale $P$-greedy algorithm and the classical multiscale method on a function with a kink.
The comparison also includes a conditioning-controlled (CC) variant, which keeps the centres of every level apart at the scale of the kernel by an exclusion radius and decides the size of every level by a certificate on the residual.
In this way it bounds the condition number of every level by construction and satisfies the two conditions of \Cref{th:stability_convergence}, see \Cref{subsec:experimental_setting}.

For the conditioning, the condition number of every level matrix is plotted together with the two bounds of \eqref{eq:general_conditioning}, the lower one in the number of centres per kernel width and the upper one in the ratio between the scale and the separation of the centres.
For the convergence, the error is plotted against the number of centres together with the reference rates, $n^{-\tau/d}$ for the $f$-greedy rule and $n^{-\tau/d+1/2}$ for quasi-uniform centres in $L^\infty(\Omega)$, and $n^{-\tau/d}$ for both in $L^2(\Omega)$.
For the cost, the measured times are set against the counts of \Cref{tab:cost}.
We also note that we omit a comparison with the classical multiscale method using globally supported kernels and samplet compression, cited in \Cref{tab:cost}, since the kernel matrices are too small for compression to provide any meaningful benefit in this setting.

All computations were performed on a MacBook Pro with an Apple M2 Max processor and 32\,GB of RAM.
The implementation is based on the \texttt{C++} library \texttt{FMCA} \cite{fmca}.

\subsection{Experimental setting}
\label{subsec:experimental_setting}

All experiments share the following setting.

\paragraph{Data.}
The only input is a finite set $X \subset \Omega$ of $N$ points, the \emph{candidates}, from which the greedy algorithms select their centres and on which all errors are measured.
On the unit square, $X$ consists of $N = 100\,000$ independent uniformly distributed points.
On the Bunny, $X$ is the point cloud itself, with $N = 300\,000$ points.

\paragraph{Level hierarchy.}
The point sets of the levels are produced by the greedy algorithm itself and are nested, $X_1 \subset X_2 \subset \dots \subset X_L \subset X$.
Level $\ell$ starts from the centres of level $\ell - 1$, whose residual it interpolates at the new scale, and keeps selecting centres until
\begin{equation}\label{eq:level_cardinality_num}
    n_\ell = 2^d n_{\ell-1},
    \qquad \text{that is} \qquad
    n_\ell = 2^{d(\ell-1)} n_1,
\end{equation}
which mimics the halving of the fill distance of a quasi-uniform set in dimension $d$.
The number of centres of the multiscale approximant is $n_L$, the size of the finest level.
This choice slightly differs from \Cref{alg:multiscale_greedy}, which returns the iterate with the smallest error.
In practice, however, the last iterate typically attains the smallest error, so we use it for simplicity.

\paragraph{Conditioning-controlled (CC) variant.}
By \eqref{eq:general_conditioning}, the condition number of a level is bounded in terms of the ratio $\delta_\ell / q_{X_\ell}$.
In the comparison of \Cref{subsec:comparison} we therefore include a variant of the multiscale greedy algorithms that bounds this ratio by construction.
On level $\ell$ the next centre is selected by the greedy rule only among the candidates farther than $\rho \delta_\ell$ from the current centres, for a fixed $\rho \in (0, 1]$.
Any two centres of the level are then at least $\rho \delta_\ell$ apart, that is
\begin{equation}\label{eq:stable_ratio}
    \frac{\delta_\ell}{q_{X_\ell}} \leq \frac{2}{\rho}
    \qquad \text{and} \qquad
    \kappa_2 \leq C_\kappa \bigg( \frac{2}{\rho} \bigg)^{2\tau}
\end{equation}
on every level, whatever the target.
This is similar in spirit to the $\gamma$-stabilized greedy algorithms of \cite{wenzel2021novel}, which restrict the selection by a criterion on the power function.
The CC variant instead controls the ratio in \eqref{eq:stable_ratio} directly.
Moreover, the size of a level is not prescribed.
The level is closed as soon as the residual $r_\ell$ left by its interpolant $s_\ell$ satisfies
\begin{align}\label{eq:stable_certificate}
    \| r_\ell \|_{L^2(\Omega)} \leq \theta \| s_\ell \|_{\Phi_{\delta_\ell}},
\end{align}
with a fixed $\theta \in (0, 1)$, or when no admissible candidate is left.
{In practice, $\Omega$ is replaced by the candidates $X$, and the $L^2(\Omega)$ norm in \eqref{eq:stable_certificate} is approximated by the corresponding discrete norm over $X$.}
The scale is then halved and the next level starts from the centres selected so far.
We use $\rho = 1/4$ and $\theta = \mu^\tau$, and the CC variant runs until it has selected as many centres as the other methods.

\paragraph{Kernel and rescaling.}
All greedy methods use the rescaled exponential kernel, that is Mat\'ern-$1/2$,
\begin{align*}
    \Phi_\delta(r) = \delta^{-d} \exp\bigg(-\frac{r}{\delta}\bigg),
\end{align*}
whose low smoothness suits the non-smooth targets considered here.
It satisfies \eqref{eq:fourier_decay} with $\tau = 3/2$ in $d = 2$, in accordance with \Cref{sec:analysis}.
On each level the kernel is rescaled according to
\begin{equation}\label{eq:delta_ell}
    \delta_\ell = \delta_1 2^{-(\ell-1)} = \nu n_\ell^{-1/d} D,
    \qquad
    \delta_1 = \nu n_1^{-1/d} D,
    \qquad
    D \isdef |\Omega|^{1/d}.
\end{equation}
This is the scaling $\delta_\ell \approx \nu h_{X_\ell}$ of the quasi-uniform case, since $n_\ell$ quasi-uniform points have fill distance $h_{X_\ell} \approx (|\Omega| / n_\ell)^{1/d} = D n_\ell^{-1/d}$.
We set $\nu = 2$ on the unit square and $\nu = 1$ on the Bunny surface.

\paragraph{Error measures.}
All errors are measured on $X$.
For an approximant $s$ of the target $f$ we report
\begin{equation*}
    \varepsilon_2 = \| f - s \|_2 / \| f \|_2,
    \qquad
    \varepsilon_\infty = \max_{\bs y \in X} |f(\bs y) - s(\bs y)|,
\end{equation*}
with the Euclidean norm over $X$, together with the spectral condition number $\kappa_2$ of the kernel matrix of the current level, which is the only matrix a multiscale method ever factorizes.

\paragraph{Implementation.}
Since the greedy selection keeps the number of centres moderate, all linear systems of the greedy methods are solved by dense Cholesky factorization.
The factorization is updated incrementally.
When the centre $\bs x_{n+1}$ is added, the kernel matrix and its Cholesky factor grow as
\begin{equation*}
    K_{n+1} = \begin{bmatrix} K_n & \bs k \\ \bs k^\top & \Phi_\delta(0) \end{bmatrix},
    \qquad
    L_{n+1} = \begin{bmatrix} L_n & \bs 0 \\ \bs l^\top & c \end{bmatrix},
\end{equation*}
where $K_n \in \mathbb{R}^{n \times n}$ is the kernel matrix on the current centres $X_n = \{\bs x_1, \dots, \bs x_n\}$ and $\bs k = \big(\Phi_\delta(\|\bs x_1 - \bs x_{n+1}\|), \dots, \Phi_\delta(\|\bs x_n - \bs x_{n+1}\|)\big)^\top$ collects the kernel values between the current centres and the new one.
From $K_{n+1} = L_{n+1} L_{n+1}^\top$ one reads off
\[
    L_n \bs l = \bs k,
    \qquad
    c = \sqrt{\Phi_\delta(0) - \| \bs l \|_2^2}.
\]
Note that $c = P_{X_n}(\bs x_{n+1})$, so the update also yields the power function at the new centre.
By \Cref{lem:uncertainty} its square bounds the smallest eigenvalue of $K_{n+1}$ from above, so the factorization monitors its own conditioning.
The classical multiscale method of \Cref{subsec:comparison} works with sparse matrices and the conjugate gradient method instead, as described there.

\subsection{\texorpdfstring{Multiscale $f$-greedy approximation on the unit square}{Multiscale f-greedy approximation on the unit square}}
\label{subsec:2d_examples}

We start with $n_1 = 5$ centres on the coarsest level and $L = 6$ levels.
\Cref{tab:levels_2d} reports the number of centres of each level and the corresponding shape parameters, following \eqref{eq:level_cardinality_num} and \eqref{eq:delta_ell}.
The same schedule serves the comparison of \Cref{subsec:comparison}.

\begin{table}[htb]
\centering
\renewcommand{\arraystretch}{1.25}
\begin{tabular}{l c c c c c c}
\hline
$\ell$ & 1 & 2 & 3 & 4 & 5 & 6 \\
\hline
$n_\ell$      & $5$ & $20$ & $80$ & $320$ & $1\,280$ & $5\,120$ \\
$\delta_\ell$ & $0.8944$ & $0.4472$ & $0.2236$ & $0.1118$ & $0.0559$ & $0.0280$ \\
\hline
\end{tabular}
\caption{Level cardinalities and kernel shape parameters of the experiments on the unit square, with $n_1 = 5$ and $\nu = 2$.}
\label{tab:levels_2d}
\end{table}

The target function has been used in \cite{li2026local,avesani2025sobolev} to test smoothness class detection and combines features of varying size, shape and regularity.
It is designed to show how the adaptive selection distributes the centres among these features.
The jump has been smoothed so that the function is continuous, and the interpolation is performed with the exponential kernel.
The function is shown in \Cref{fig:composite_function}.

\begin{figure}[htb]
  \centering
  \begin{tikzpicture}
    \node[anchor=south west, inner sep=0] (img) at (0,0)
      {\includegraphics[width=0.36\textwidth]{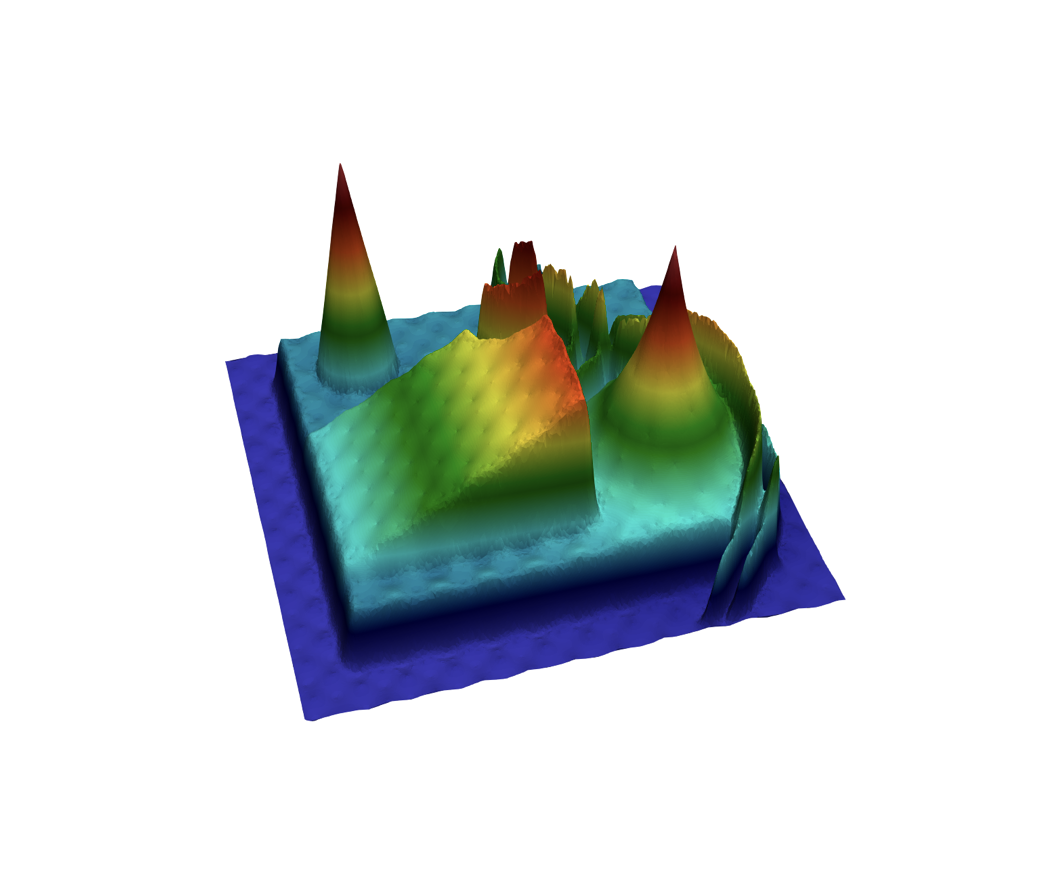}};
    \begin{scope}[
        x={(img.south east)},
        y={(img.north west)},
        font=\small,
        every path/.style={-{Stealth[length=5pt]}, thick},
        every node/.style={fill=white, fill opacity=0.75,
                           text opacity=1, inner sep=2pt,
                           rounded corners=2pt}
    ]
      \node (cone) at (0.05, 0.70) {cone peak};
      \draw (cone) -- (0.31, 0.70);
      \node (ridge) at (0.72, 0.85) {oscillatory wedge};
      \draw (ridge) -- (0.53, 0.72);
      \node (cusp) at (0.75, 0.17) {exponential cusp};
      \draw (cusp) -- (0.65, 0.55);
      \node (wedge) at (0.98, 0.65) {curved ridge};
      \draw (wedge) -- (0.73, 0.55);
      \node (grad) at (0.00, 0.35) {linear gradient};
      \draw (grad) -- (0.40, 0.45);
      \node (block) at (0.28, 0.95) {constant block};
      \draw (block) -- (0.42, 0.63);
        \node (smooth) at (0.10, 0.10) {smooth transition};
      \draw (smooth) -- (0.315, 0.28);
    \end{scope}
  \end{tikzpicture}
  \caption{Composite target function with localized features of varying size, shape and regularity.}\label{fig:composite_function}
\end{figure}
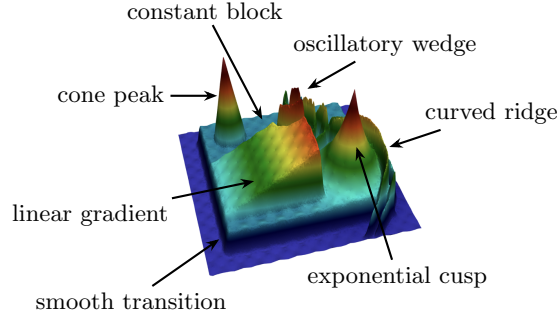

\Cref{fig:composite_points_levels} shows the centres selected by the multiscale $f$-greedy method at levels $\ell = 2, 4$ and $6$, and \Cref{fig:composite_sol} the corresponding interpolants.
The errors and condition numbers are reported in \Cref{tab:error_composite}.
The error decreases more rapidly after level $4$.
Up to level $4$ one has $\delta_\ell \geq 0.1118$, which is larger than all the localized features of \Cref{fig:composite_function}.
From level $5$ on, $\delta_\ell$ becomes comparable to the size of these features, so that the adaptive selection can resolve them.
The condition numbers grow by two orders of magnitude over the six levels, which is consistent with the reduced smoothness of the exponential kernel and the exponent $2\tau$ in \Cref{prop:stability}.

\begin{figure}[!htbp]
\centering
    \includegraphics[width=0.26\linewidth]{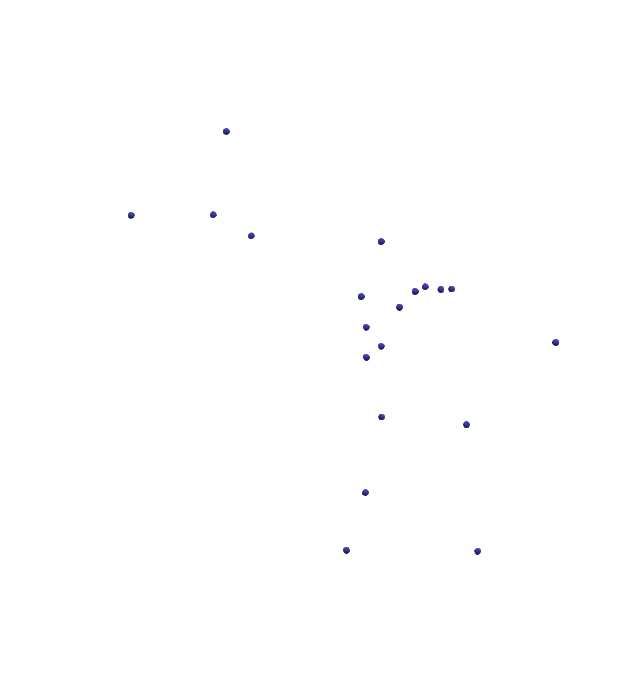}
    \includegraphics[width=0.26\linewidth]{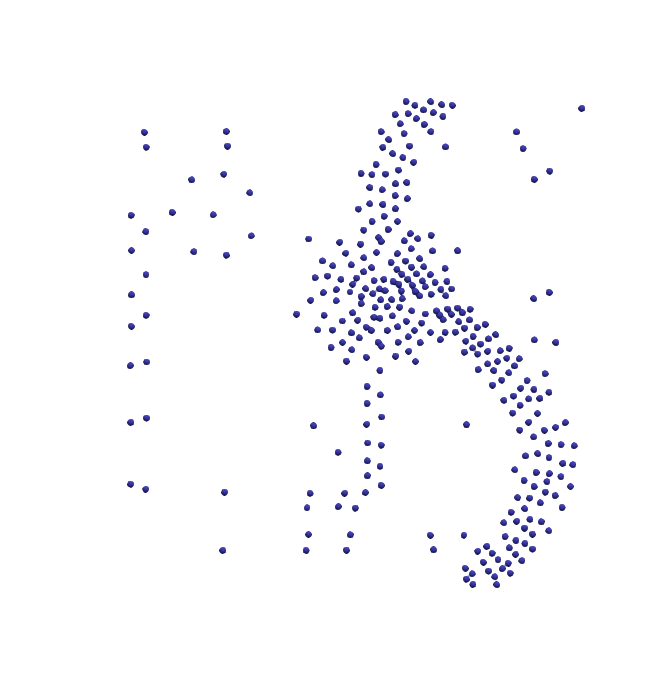}
    \includegraphics[width=0.374\linewidth]{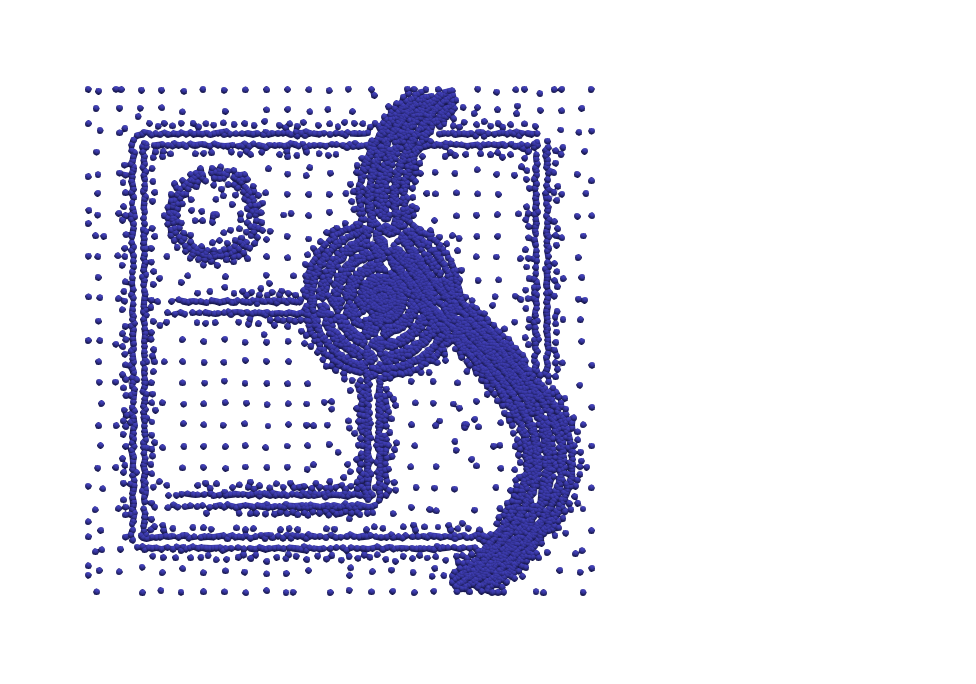}
    \caption{Centres selected by multiscale $f$-greedy for the target function
    in \Cref{fig:composite_function} for levels $\ell = 2, 4, 6$.}
    \label{fig:composite_points_levels}
\end{figure}

\begin{figure}[!htbp]
\centering
    \includegraphics[width=0.26\linewidth]{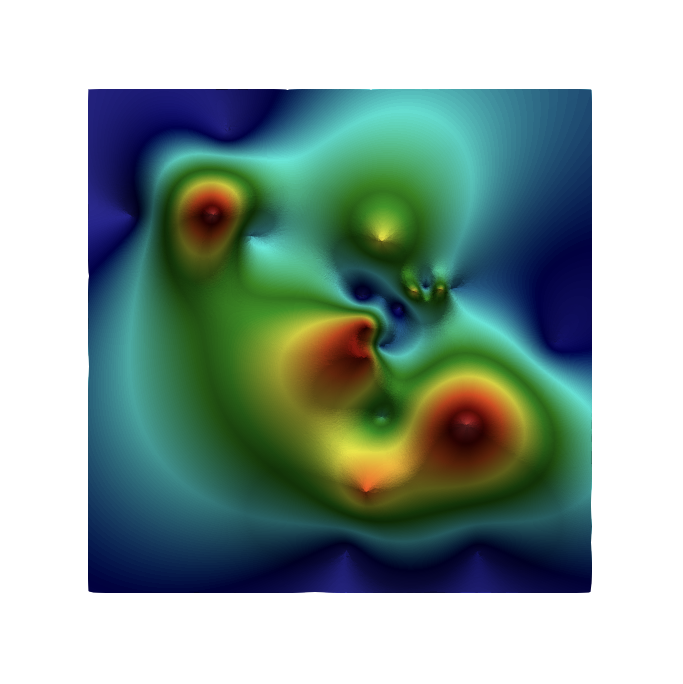}
    \includegraphics[width=0.25\linewidth]{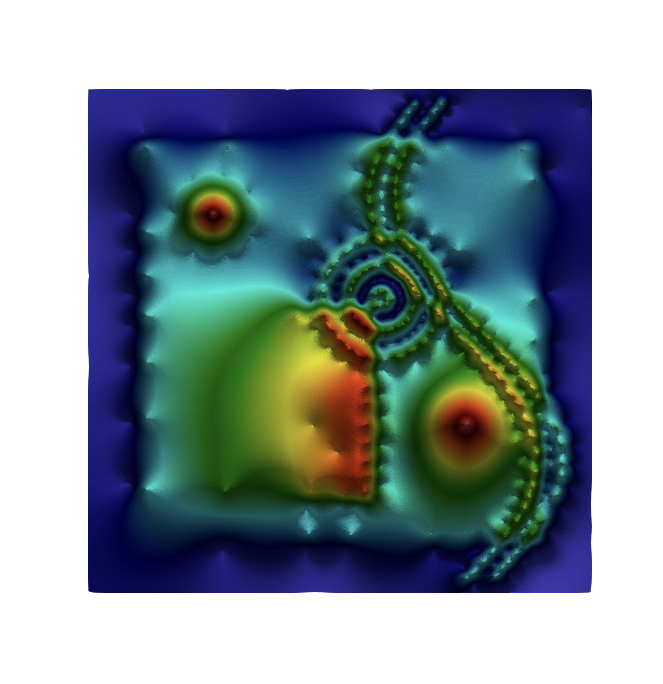}
    \includegraphics[width=0.374\linewidth]{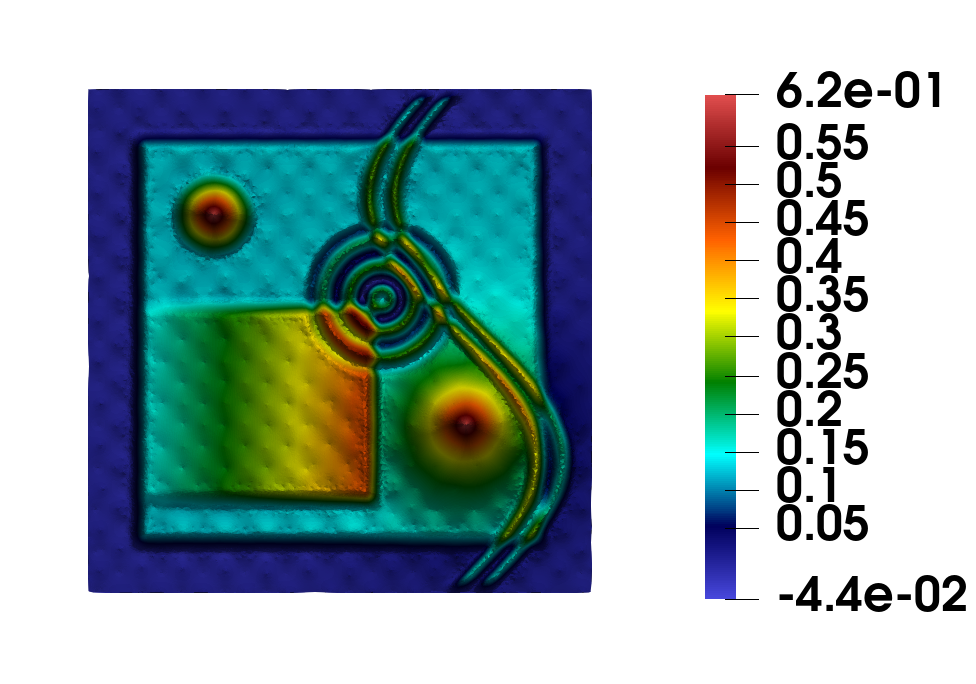}
    \caption{Interpolation results for multiscale $f$-greedy for the target
    function in \Cref{fig:composite_function} for levels $\ell = 2, 4, 6$.}
    \label{fig:composite_sol}
\end{figure}

\begin{table}[htb]
\centering
\renewcommand{\arraystretch}{1.25}
\begin{tabular}{ c c c }
\hline
$\ell$ & $\varepsilon_\infty$ & $\kappa_2$ \\
\hline
1 & $4.12 \cdot 10^{-1}$ & $2.15 \cdot 10^{1}$ \\
2 & $2.29 \cdot 10^{-1}$ & $3.25 \cdot 10^{2}$ \\
3 & $1.34 \cdot 10^{-1}$ & $6.13 \cdot 10^{2}$ \\
4 & $7.77 \cdot 10^{-2}$ & $8.19 \cdot 10^{2}$ \\
5 & $1.80 \cdot 10^{-2}$ & $1.31 \cdot 10^{3}$ \\
6 & $3.62 \cdot 10^{-3}$ & $2.10 \cdot 10^{3}$ \\
\hline
\end{tabular}
\caption{Approximation of the target function in \Cref{fig:composite_function} with the exponential kernel: error and condition number of the kernel matrix at the end of each level, for the multiscale $f$-greedy algorithm.}
\label{tab:error_composite}
\end{table}

\subsection{\texorpdfstring{Multiscale $f$-greedy approximation on a surface}{Multiscale f-greedy approximation on a surface}}
\label{subsec:bunny}

In this experiment we consider scattered data on the Stanford Bunny surface from the \texttt{Stanford 3D Scanning Repository}\footnote{\url{https://graphics.stanford.edu/data/3Dscanrep/}}, discretized by $N = 300\,000$ points.
Although the theoretical results are formulated for Euclidean domains, this example shows that the methods also work on manifolds.

The target function combines a smooth localized feature and a non-smooth one,
\begin{equation}\label{eq:bunny_function}
    f(\bs x) = \alpha |x_1 - c|
             + \exp\Big(-\frac{\|\bs x - \bs c_g\|^2}{2\sigma_g^2}\Big),
\end{equation}
with $c = -0.080$, $\alpha = 7$, $\bs c_g = (0.019, 0.084, 0.051)$ and $\sigma_g = 0.015$.
The first term is a classical corner.
Its gradient jumps across the plane $x_1 = c$, which intersects the Bunny along a closed curve running across the head and the chest, 
while away from that curve the target is simply linear in $x_1$.
The second term is a Gaussian bump located on the back of the Bunny.
The target is therefore smooth except along one curve of the manifold, where it is only Lipschitz continuous.
Since $f \notin C^1(\Omega)$, we employ the exponential kernel.

The Bunny is a two-dimensional manifold embedded in $\mathbb{R}^3$, so the dimension entering \eqref{eq:level_cardinality_num} and \eqref{eq:delta_ell} is the \emph{intrinsic} one, $d = 2$.
We therefore use $n_\ell = 4 n_{\ell-1}$ with $n_1 = 9$ and $L = 6$ levels.
The resulting level cardinalities and shape parameters are collected in \Cref{tab:levels_bunny}.

\begin{table}[htbp]
\centering
\renewcommand{\arraystretch}{1.25}
\begin{tabular}{l c c c c c c }
\hline
$\ell$ & 1 & 2 & 3 & 4 & 5 & 6\\
\hline
$n_\ell$      & $9$ & $36$ & $144$ & $576$ & $2\,304$ & $9\,216$\\
$\delta_\ell$ & $0.0811$ & $0.0406$ & $0.0203$ & $0.0101$ & $0.0051$ & $0.0025$ \\
\hline
\end{tabular}
\caption{Level cardinalities and kernel shape parameters for the Bunny experiment, corresponding to $n_1 = 9$, intrinsic dimension $d = 2$ and $\nu = 1$.}
\label{tab:levels_bunny}
\end{table}

The selected centres and the corresponding interpolants are shown in \Cref{fig:bunny_points_sol}, the errors and condition numbers are reported in \Cref{tab:error_bunny}.
The multiscale $f$-greedy algorithm captures both features of \eqref{eq:bunny_function} adaptively.
The coarser levels place a few centres over the whole surface, whereas the finer levels, with progressively smaller shape parameters, concentrate the new centres around the Gaussian bump and, mostly, along the corner.

\begin{figure}[!htbp]
    \centering
    \begin{subfigure}[b]{\textwidth}
        \centering
        \includegraphics[width=0.21\linewidth]{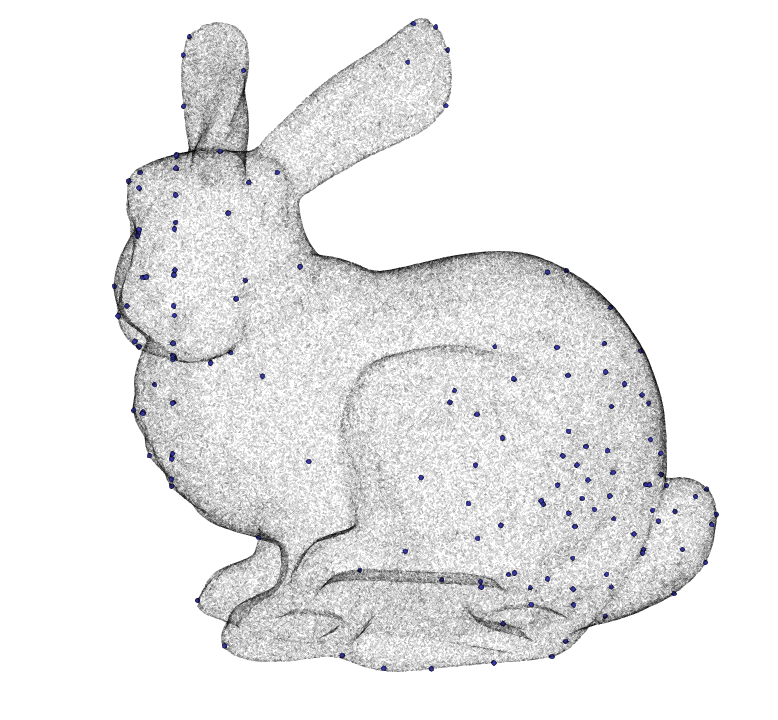} \quad \quad
        \includegraphics[width=0.21\linewidth]{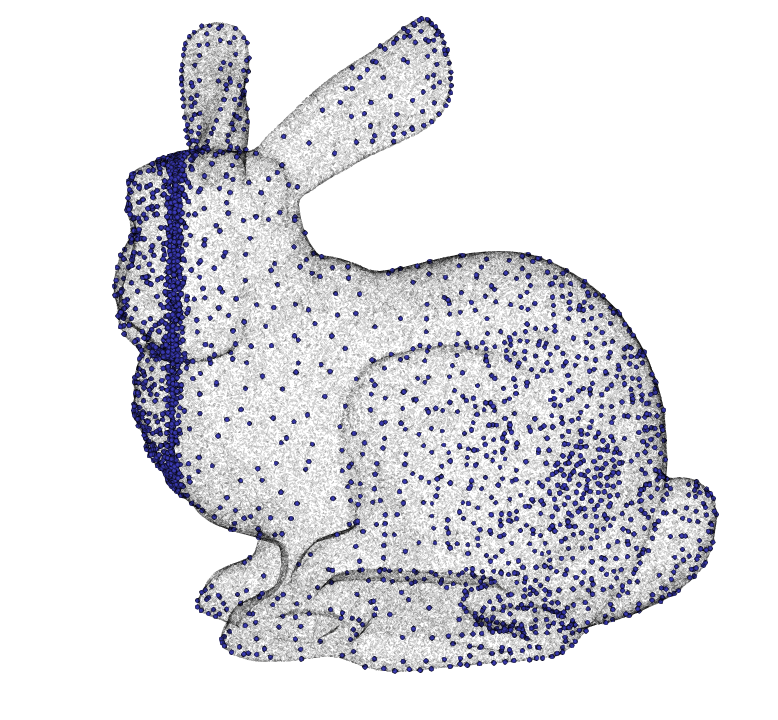}
        \quad \quad
        \includegraphics[width=0.21\linewidth]{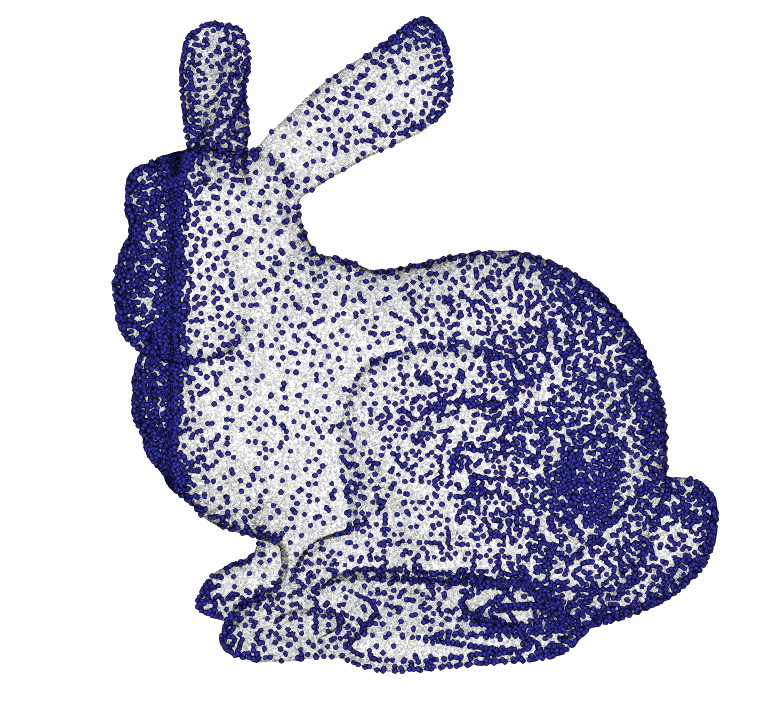}
        \quad
        \includegraphics[width=0.1\linewidth]{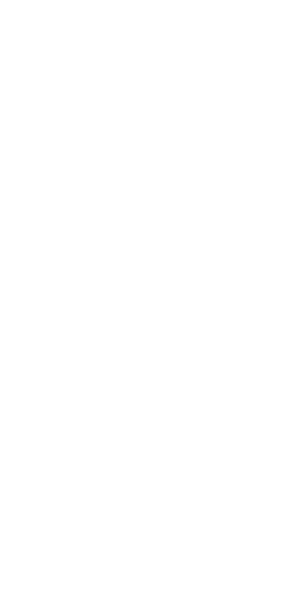}
    \end{subfigure}

    \vspace{0.5cm}

    \begin{subfigure}[b]{\textwidth}
        \centering
        \includegraphics[width=0.21\linewidth]{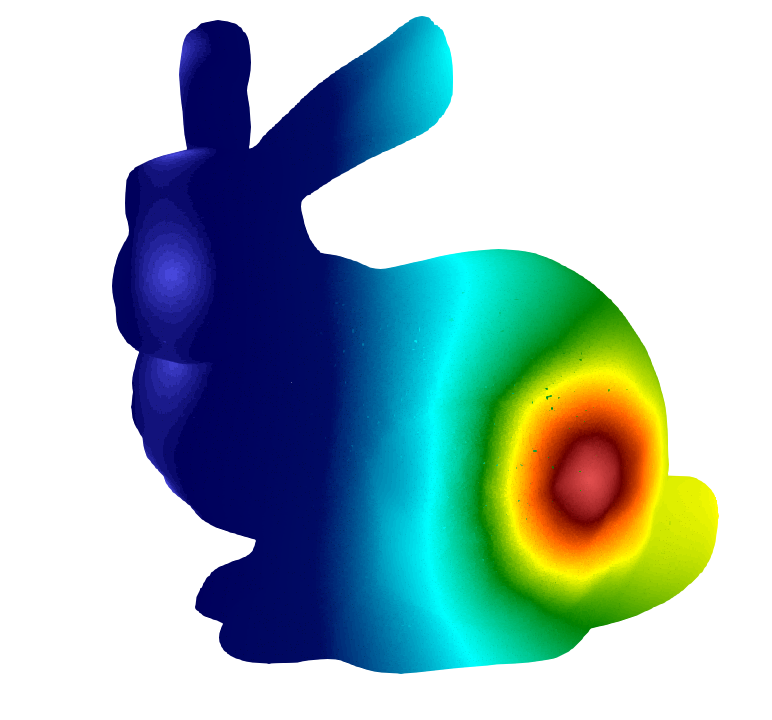}
        \quad \quad
        \includegraphics[width=0.21\linewidth]{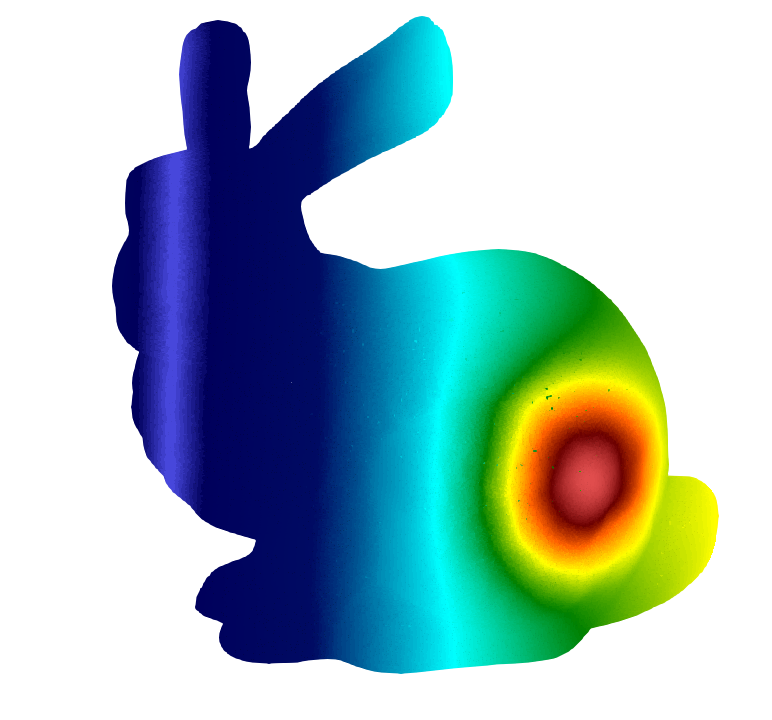}
        \quad \quad
        \includegraphics[width=0.21\linewidth]{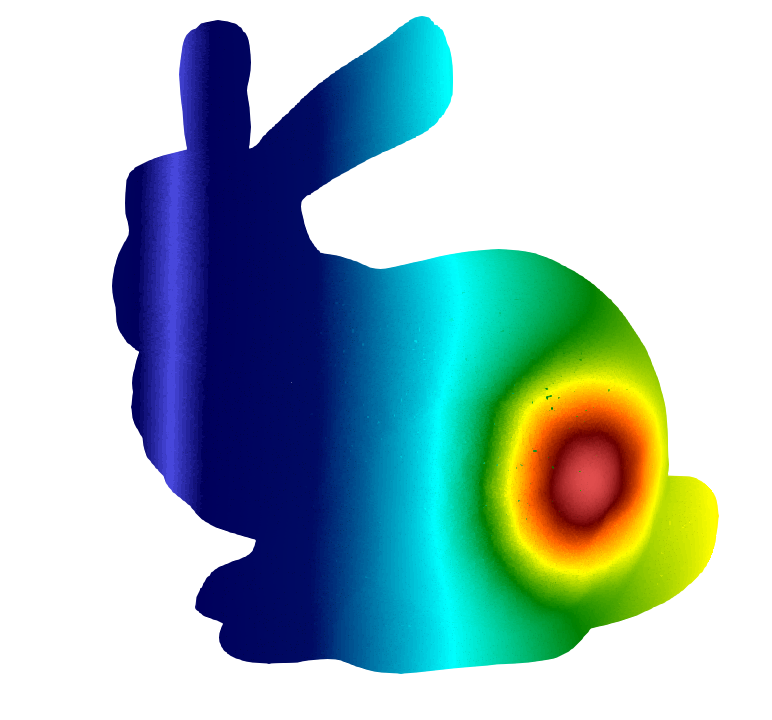}
        \quad
        \includegraphics[width=0.1\linewidth]{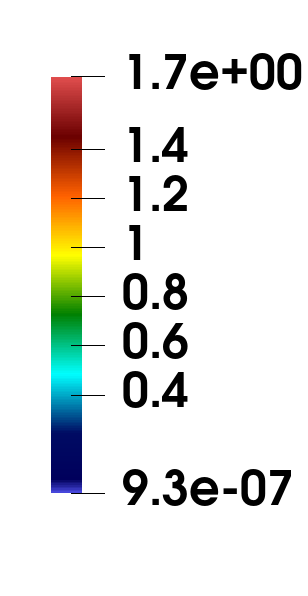}
    \end{subfigure}

    \caption{Approximation of \eqref{eq:bunny_function} on the Bunny point cloud by multiscale $f$-greedy. Top: centres selected on levels $\ell = 2, 4, 6$. Bottom: the corresponding multiscale approximants.}
    \label{fig:bunny_points_sol}
\end{figure}

\begin{table}[htb]
\centering
\renewcommand{\arraystretch}{1.25}
\begin{tabular}{ c c c }
\hline
$\ell$ & $\varepsilon_\infty$ & $\kappa_2$ \\
\hline
1 & $1.37 \cdot 10^{-1}$ & $5.12 \cdot 10^{1}$ \\
2 & $5.10 \cdot 10^{-2}$ & $7.37 \cdot 10^{1}$ \\
3 & $1.43 \cdot 10^{-2}$ & $1.04 \cdot 10^{2}$ \\
4 & $4.17 \cdot 10^{-3}$ & $1.55 \cdot 10^{2}$ \\
5 & $1.03 \cdot 10^{-3}$ & $5.97 \cdot 10^{2}$ \\
6 & $2.45 \cdot 10^{-4}$ & $1.17 \cdot 10^{3}$ \\
\hline
\end{tabular}
\caption{Approximation of \eqref{eq:bunny_function} on the Bunny point cloud with the exponential kernel: error and condition number of the kernel matrix at the end of each level, for the multiscale $f$-greedy algorithm.}
\label{tab:error_bunny}
\end{table}

\subsection{Comparison of the methods}
\label{subsec:comparison}

The target function is
\begin{equation}\label{eq:abs_function}
    f(\bs x) = |x_1 + x_2 - 1|,
    \qquad \bs x \in [0,1]^2,
\end{equation}
which is linear on both sides of the diagonal $x_1 + x_2 = 1$ and has a kink across it.
Since $f \in H^{3/2-\epsilon}(\Omega)$ for every $\epsilon > 0$ but $f \notin H^{3/2}(\Omega)$, it lies just outside the native space of the exponential kernel. Thus, we use the rates of \Cref{th:stability_convergence} as reference rates.
The kernel is the exponential one and the level setting is that of \Cref{tab:levels_2d}, so that every greedy method ends with $n = 5\,120$ centres selected from the same $N = 100\,000$ candidates.
\paragraph{Multiscale against single scale.}
\Cref{tab:error_abs} reports the errors and condition numbers at the end of every level for the multiscale greedy algorithms.
The multiscale $f$-greedy approximant reaches $\varepsilon_\infty = 1.48 \cdot 10^{-4}$ after six levels, against $8.27 \cdot 10^{-3}$ for the multiscale $P$-greedy.
The condition numbers behave as \Cref{rem:stability_convergence_cases} predicts.
For $P$-greedy the centres of every level are quasi-uniform, the ratio $\delta_\ell / q_{X_\ell}$ stays between $2.6$ and $7.2$, and the condition numbers saturate at about $10^2$.
For $f$-greedy the centres cluster along the kink, $\delta_\ell / q_{X_\ell}$ grows, and the condition numbers grow to almost $10^{4}$.

\begin{table}[htb]
\centering
\renewcommand{\arraystretch}{1.25}
\begin{tabular}{ c c c c c c c }
\hline
 & & \multicolumn{2}{c}{multiscale $f$-greedy}
 & & \multicolumn{2}{c}{multiscale $P$-greedy} \\
\cline{3-4}\cline{6-7}
$\ell$ & $n_\ell$ & $\varepsilon_\infty$ & $\kappa_2$
       & & $\varepsilon_\infty$ & $\kappa_2$ \\
\hline
1 & $5$      & $1.93 \cdot 10^{-1}$ & $6.58 \cdot 10^{0}$ & & $1.78 \cdot 10^{-1}$ & $5.76 \cdot 10^{0}$ \\
2 & $20$     & $6.17 \cdot 10^{-2}$ & $3.83 \cdot 10^{1}$ & & $1.12 \cdot 10^{-1}$ & $2.10 \cdot 10^{1}$ \\
3 & $80$     & $1.83 \cdot 10^{-2}$ & $1.14 \cdot 10^{2}$ & & $5.50 \cdot 10^{-2}$ & $5.20 \cdot 10^{1}$ \\
4 & $320$    & $5.63 \cdot 10^{-3}$ & $4.61 \cdot 10^{2}$ & & $2.91 \cdot 10^{-2}$ & $9.88 \cdot 10^{1}$ \\
5 & $1\,280$ & $1.15 \cdot 10^{-3}$ & $1.91 \cdot 10^{3}$ & & $1.50 \cdot 10^{-2}$ & $1.17 \cdot 10^{2}$ \\
6 & $5\,120$ & $1.48 \cdot 10^{-4}$ & $7.13 \cdot 10^{3}$ & & $8.27 \cdot 10^{-3}$ & $1.31 \cdot 10^{2}$ \\
\hline
\end{tabular}
\caption{Approximation of \eqref{eq:abs_function} with the exponential kernel: error and condition number of the kernel matrix at the end of each level, for the multiscale $f$-greedy and $P$-greedy algorithms.}
\label{tab:error_abs}
\end{table}

\Cref{fig:abs_error} sets the multiscale runs against the single-scale ones.
For each of the six shape parameters of \Cref{tab:levels_2d} the single-scale algorithm selects $5\,120$ centres at that fixed scale, and its error and condition number are traced against the number of centres, together with those of the multiscale run.
The trade-off of \Cref{th:tradeoff} is visible in both panels.
As $\delta$ decreases, the error of the single-scale runs increases and their condition number decreases,  see \Cref{fig:abs_cond}.
The sawtooth pattern of the multiscale condition numbers in \Cref{fig:abs_cond} reflects the hierarchical level structure of the method: the condition number grows within a level and drops at the start of the next one, where the scale halves.
In the rest of the comparison the single-scale algorithms are run using $\delta_1$.

\definecolor{s1}{RGB}{158,202,225}
\definecolor{s2}{RGB}{107,174,214}
\definecolor{s3}{RGB}{66,146,198}
\definecolor{s4}{RGB}{33,113,181}
\definecolor{s5}{RGB}{8,81,156}
\definecolor{s6}{RGB}{8,48,107}
\colorlet{msc}{red}              %
\colorlet{stc}{blue}             %
\colorlet{clc}{green!55!black}   %
\pgfplotsset{
  greedycmp/.style={
    width=.5\textwidth, height=6.2cm,
    xlabel={number of centres},
    grid=both, grid style={gray!20},
    legend style={font=\small, draw=none, fill=none},
    legend cell align=left,
    every axis plot/.append style={line width=.7pt},
  },
  condcmp/.style={greedycmp, width=.33\textwidth, height=5.2cm, ymin=1},
}
\newcommand{\czeroM}{0.72}
\newcommand{\ckappaM}{1.78}
\newcommand{\abstrace}[1]{Figures/trace_abs_#1.txt}

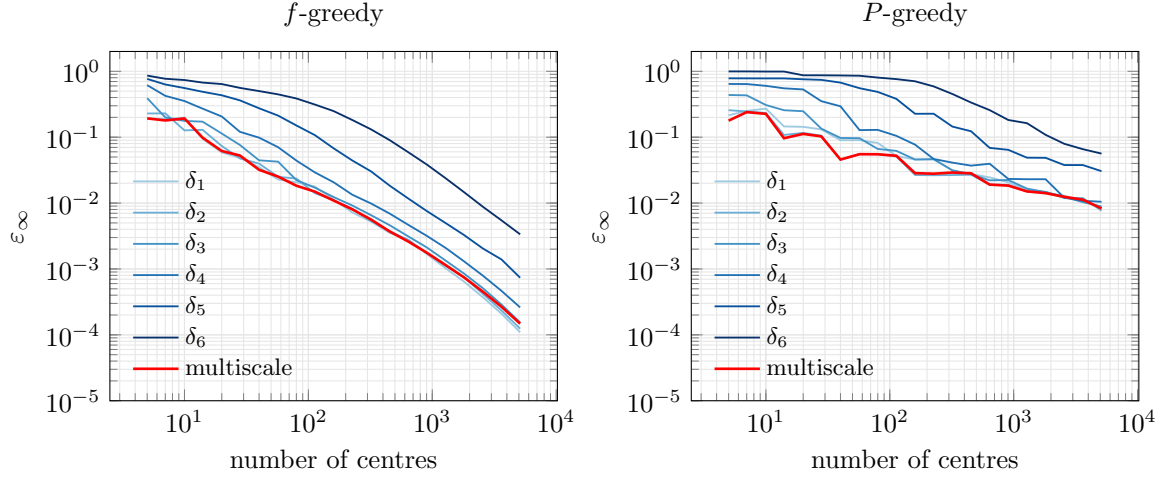
\begin{figure}[!htbp]
\centering
\begin{tikzpicture}
\begin{loglogaxis}[greedycmp, ylabel={$\varepsilon_\infty$}, title={$f$-greedy},
                   legend pos=south west, ymin=1e-5, ymax=2e0]
\addplot[s1] table[skip first n=1, x=n, y=linf] {\abstrace{fgreedy_1}}; \addlegendentry{$\delta_1$}
\addplot[s2] table[skip first n=1, x=n, y=linf] {\abstrace{fgreedy_2}}; \addlegendentry{$\delta_2$}
\addplot[s3] table[skip first n=1, x=n, y=linf] {\abstrace{fgreedy_3}}; \addlegendentry{$\delta_3$}
\addplot[s4] table[skip first n=1, x=n, y=linf] {\abstrace{fgreedy_4}}; \addlegendentry{$\delta_4$}
\addplot[s5] table[skip first n=1, x=n, y=linf] {\abstrace{fgreedy_5}}; \addlegendentry{$\delta_5$}
\addplot[s6] table[skip first n=1, x=n, y=linf] {\abstrace{fgreedy_6}}; \addlegendentry{$\delta_6$}
\addplot[msc, line width=1pt] table[skip first n=1, x=n, y=linf] {\abstrace{ms_fgreedy}}; \addlegendentry{multiscale}
\end{loglogaxis}
\end{tikzpicture}%
\begin{tikzpicture}
\begin{loglogaxis}[greedycmp, ylabel={$\varepsilon_\infty$}, title={$P$-greedy},
                   legend pos=south west, ymin=1e-5, ymax=2e0]
\addplot[s1] table[skip first n=1, x=n, y=linf] {\abstrace{pgreedy_1}}; \addlegendentry{$\delta_1$}
\addplot[s2] table[skip first n=1, x=n, y=linf] {\abstrace{pgreedy_2}}; \addlegendentry{$\delta_2$}
\addplot[s3] table[skip first n=1, x=n, y=linf] {\abstrace{pgreedy_3}}; \addlegendentry{$\delta_3$}
\addplot[s4] table[skip first n=1, x=n, y=linf] {\abstrace{pgreedy_4}}; \addlegendentry{$\delta_4$}
\addplot[s5] table[skip first n=1, x=n, y=linf] {\abstrace{pgreedy_5}}; \addlegendentry{$\delta_5$}
\addplot[s6] table[skip first n=1, x=n, y=linf] {\abstrace{pgreedy_6}}; \addlegendentry{$\delta_6$}
\addplot[msc, line width=1pt] table[skip first n=1, x=n, y=linf] {\abstrace{ms_pgreedy}}; \addlegendentry{multiscale}
\end{loglogaxis}
\end{tikzpicture}
\caption{Approximation of \eqref{eq:abs_function}: decay of $\varepsilon_\infty$ against the number of centres for the single-scale runs at the six shape parameters of \Cref{tab:levels_2d} and for the multiscale run. Left: $f$-greedy. Right: $P$-greedy.}
\label{fig:abs_error}
\end{figure}

\begin{figure}[!htbp]
\centering
\begin{tikzpicture}
\begin{loglogaxis}[greedycmp, ylabel={$\kappa_2$}, title={$f$-greedy},
                   legend pos=north west, ymin=5e-1, ymax=1e8]
\addplot[s1] table[skip first n=1, x=n, y=cond] {\abstrace{fgreedy_1}}; \addlegendentry{$\delta_1$}
\addplot[s2] table[skip first n=1, x=n, y=cond] {\abstrace{fgreedy_2}}; \addlegendentry{$\delta_2$}
\addplot[s3] table[skip first n=1, x=n, y=cond] {\abstrace{fgreedy_3}}; \addlegendentry{$\delta_3$}
\addplot[s4] table[skip first n=1, x=n, y=cond] {\abstrace{fgreedy_4}}; \addlegendentry{$\delta_4$}
\addplot[s5] table[skip first n=1, x=n, y=cond] {\abstrace{fgreedy_5}}; \addlegendentry{$\delta_5$}
\addplot[s6] table[skip first n=1, x=n, y=cond] {\abstrace{fgreedy_6}}; \addlegendentry{$\delta_6$}
\addplot[msc, line width=1pt] table[skip first n=1, x=n, y=cond] {\abstrace{ms_fgreedy}}; \addlegendentry{multiscale}
\end{loglogaxis}
\end{tikzpicture}%
\begin{tikzpicture}
\begin{loglogaxis}[greedycmp, ylabel={$\kappa_2$}, title={$P$-greedy},
                   legend pos=north west, ymin=5e-1, ymax=1e8]
\addplot[s1] table[skip first n=1, x=n, y=cond] {\abstrace{pgreedy_1}}; \addlegendentry{$\delta_1$}
\addplot[s2] table[skip first n=1, x=n, y=cond] {\abstrace{pgreedy_2}}; \addlegendentry{$\delta_2$}
\addplot[s3] table[skip first n=1, x=n, y=cond] {\abstrace{pgreedy_3}}; \addlegendentry{$\delta_3$}
\addplot[s4] table[skip first n=1, x=n, y=cond] {\abstrace{pgreedy_4}}; \addlegendentry{$\delta_4$}
\addplot[s5] table[skip first n=1, x=n, y=cond] {\abstrace{pgreedy_5}}; \addlegendentry{$\delta_5$}
\addplot[s6] table[skip first n=1, x=n, y=cond] {\abstrace{pgreedy_6}}; \addlegendentry{$\delta_6$}
\addplot[msc, line width=1pt] table[skip first n=1, x=n, y=cond] {\abstrace{ms_pgreedy}}; \addlegendentry{multiscale}
\end{loglogaxis}
\end{tikzpicture}
\caption{Approximation of \eqref{eq:abs_function}: condition number $\kappa_2$ of the kernel matrix against the number of centres, for the single-scale runs at the six shape parameters of \Cref{tab:levels_2d} and for the multiscale run, whose $\kappa_2$ is that of the current level. Left: $f$-greedy. Right: $P$-greedy.}
\label{fig:abs_cond}
\end{figure}
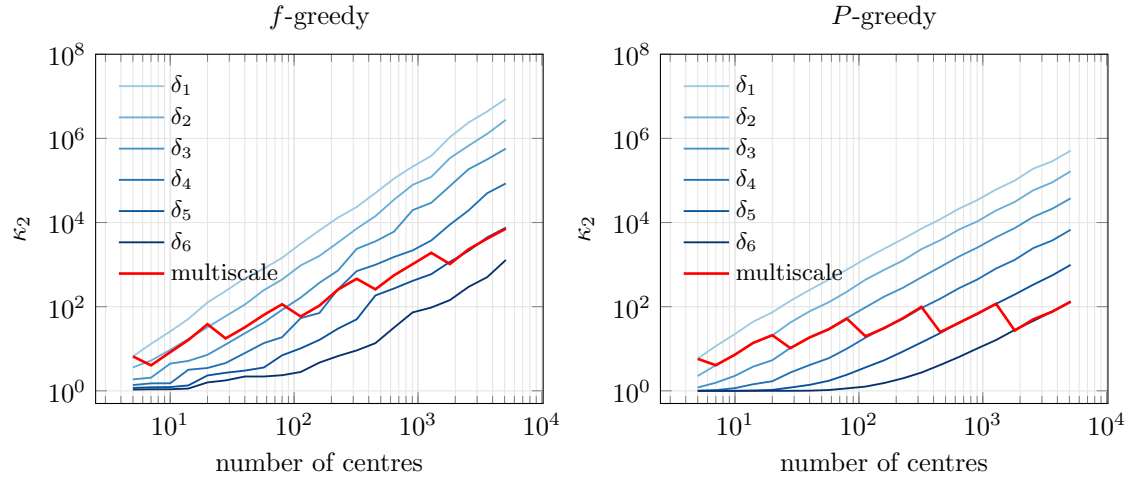

\paragraph{The other approaches.}
The conditioning-controlled variant runs with $\rho = 1/4$ and $\theta = \mu^{\tau} =  2^{-3/2}$, so that $\delta_\ell / q_{X_\ell} \leq 8$ on every level by \eqref{eq:stable_ratio}.
The errors and condition numbers are reported in \Cref{tab:error_abs_stable}.

The classical multiscale method works on nested quasi-uniform subsets of $X$: level $\ell$ keeps the points of level $\ell - 1$ and adds the candidate closest to the centre of every cell of a $2^{\ell+1} \times 2^{\ell+1}$ grid, thinned to a separation of half a cell.
This gives uniform refinement within the levels.
The kernel is the Wendland function $\phi_{2,0}(r) = (1 - r)_+^2$, which is continuous with $\tau = 3/2$ in $d = 2$, the same smoothness as the exponential kernel, with support radius $2 \delta_\ell$.
The level systems are sparse and are solved by the conjugate gradient method to a relative residual of $10^{-10}$.
Level $5$, with $5\,297$ points, is the one comparable to the greedy budget.
Since the cost per centre of this method is small, the hierarchy is continued to level $7$.

\begin{table}[htb]
\centering
\renewcommand{\arraystretch}{1.25}
\begin{tabular}{ c c c c c c c c }
\hline
 & \multicolumn{3}{c}{CC $f$-greedy}
 & & \multicolumn{3}{c}{CC $P$-greedy} \\
\cline{2-4}\cline{6-8}
$\ell$ & $n_\ell$ & $\varepsilon_\infty$ & $\kappa_2$
       & & $n_\ell$ & $\varepsilon_\infty$ & $\kappa_2$ \\
\hline
1 & $1$      & $7.84 \cdot 10^{-1}$ & $1.00 \cdot 10^{0}$ & & $3$      & $9.89 \cdot 10^{-1}$ & $2.95 \cdot 10^{0}$ \\
2 & $3$      & $3.13 \cdot 10^{-1}$ & $1.81 \cdot 10^{0}$ & & $4$      & $2.77 \cdot 10^{-1}$ & $1.86 \cdot 10^{0}$ \\
3 & $10$     & $1.36 \cdot 10^{-1}$ & $3.93 \cdot 10^{0}$ & & $17$     & $1.73 \cdot 10^{-1}$ & $4.39 \cdot 10^{0}$ \\
4 & $37$     & $3.36 \cdot 10^{-2}$ & $4.99 \cdot 10^{0}$ & & $72$     & $8.72 \cdot 10^{-2}$ & $8.59 \cdot 10^{0}$ \\
5 & $185$    & $1.06 \cdot 10^{-2}$ & $2.31 \cdot 10^{1}$ & & $335$    & $3.83 \cdot 10^{-2}$ & $1.54 \cdot 10^{1}$ \\
6 & $757$    & $5.24 \cdot 10^{-3}$ & $6.84 \cdot 10^{1}$ & & $1\,295$ & $2.01 \cdot 10^{-2}$ & $1.68 \cdot 10^{1}$ \\
7 & $2\,759$ & $2.38 \cdot 10^{-3}$ & $9.61 \cdot 10^{1}$ & & $5\,120$ & $9.04 \cdot 10^{-3}$ & $1.73 \cdot 10^{1}$ \\
8 & $5\,120$ & $1.31 \cdot 10^{-3}$ & $6.55 \cdot 10^{1}$ & & --       & --                   & --                  \\
\hline
\end{tabular}
\caption{Approximation of \eqref{eq:abs_function} by the conditioning-controlled variant with $\rho = 1/4$ and $\theta = 2^{-3/2}$: size of the levels decided by the certificate, error and condition number of the kernel matrix at the end of each level, for the $f$-greedy and the $P$-greedy rule.}
\label{tab:error_abs_stable}
\end{table}

\paragraph{Conditioning.}
\Cref{fig:abs_comparison_cond} shows, for every method, the condition number of every kernel matrix, together with the two bounds of \eqref{eq:general_conditioning} for that matrix, which read
\[
    c_0 \big( \delta_\ell n_\ell^{1/d} \big)^{2\tau}
    \leq \kappa_2
    \leq C_\kappa \max\{1, \delta_\ell / q_{X_\ell}\}^{2\tau},
\]
where $n_\ell$ is the current number of centres of the level and $q_{X_\ell}$ their separation radius.
The constants $c_0$ and $C_\kappa$ depend only on the kernel, the dimension and the domain, but the theory does not give their values.
We therefore take $c_0 = 0.72$ and $C_\kappa = 1.78$, the largest and the smallest values for which the bounds hold for all the matrices of all the methods.
The figure shows the cases of \Cref{rem:stability_convergence_cases}.
At a fixed scale the lower bound grows like $n_\ell^{2\tau/d}$.
For the multiscale methods with prescribed level sizes, $\delta_\ell n_\ell^{1/d} = \nu$ at the end of every level, so the lower bound is the same on every level.
The condition numbers of the multiscale $f$-greedy grow with $\delta_\ell / q_{X_\ell}$, together with the upper bound, and remain three orders of magnitude below the single scale.
The exclusion radius of the CC variant keeps $\delta_\ell / q_{X_\ell} \leq 8$, so the upper bound stays below $C_\kappa 8^{2\tau}$ and the condition numbers stay below $10^2$ on all levels, with both rules.

\begin{figure}[!htbp]
\centering
\begin{tikzpicture}
\begin{loglogaxis}[condcmp, xlabel={}, ylabel={$\kappa_2$}, title={$f$-greedy, $\delta_1$}, ymin=1e-2, ymax=1e12, legend pos=north west]
\addplot[black] table[skip first n=1, x=n, y=cond] {\abstrace{fgreedy_1}}; \addlegendentry{$\kappa_2$}
\addplot[gray, dashed] table[skip first n=1, x=n, y expr=\ckappaM*\thisrowno{8}] {\abstrace{fgreedy_1}};
\addplot[gray, dotted, line width=1pt] table[skip first n=1, x=n, y expr=\czeroM*\thisrowno{7}] {\abstrace{fgreedy_1}};
\end{loglogaxis}
\end{tikzpicture}%
\begin{tikzpicture}
\begin{loglogaxis}[condcmp, xlabel={}, title={multiscale $f$-greedy}, ymin=1e-2, ymax=1e12]
\addplot[msc] table[skip first n=1, x=n, y=cond] {\abstrace{ms_fgreedy}};
\addplot[gray, dashed] table[skip first n=1, x=n, y expr=\ckappaM*\thisrowno{8}] {\abstrace{ms_fgreedy}};
\addplot[gray, dotted, line width=1pt] table[skip first n=1, x=n, y expr=\czeroM*\thisrowno{7}] {\abstrace{ms_fgreedy}};
\end{loglogaxis}
\end{tikzpicture}%
\begin{tikzpicture}
\begin{loglogaxis}[condcmp, xlabel={}, title={CC $f$-greedy}, ymin=1e-2, ymax=1e12,
                   xmin=0.7, xmax=1e4, xtickten={0,...,4}, xticklabels={,$10^{1}$,$10^{2}$,$10^{3}$,$10^{4}$},]
\addplot[stc] table[skip first n=1, x=n, y=cond] {\abstrace{stable_fgreedy}};
\addplot[gray, dashed] table[skip first n=1, x=n, y expr=\ckappaM*\thisrowno{8}] {\abstrace{stable_fgreedy}};
\addplot[gray, dotted, line width=1pt] table[skip first n=1, x=n, y expr=\czeroM*\thisrowno{7}] {\abstrace{stable_fgreedy}};
\end{loglogaxis}
\end{tikzpicture}\\
\begin{tikzpicture}
\begin{loglogaxis}[condcmp, ylabel={$\kappa_2$}, title={$P$-greedy, $\delta_1$}, ymin=1e-2, ymax=1e12]
\addplot[black] table[skip first n=1, x=n, y=cond] {\abstrace{pgreedy_1}};
\addplot[gray, dashed] table[skip first n=1, x=n, y expr=\ckappaM*\thisrowno{8}] {\abstrace{pgreedy_1}};
\addplot[gray, dotted, line width=1pt] table[skip first n=1, x=n, y expr=\czeroM*\thisrowno{7}] {\abstrace{pgreedy_1}};
\end{loglogaxis}
\end{tikzpicture}%
\begin{tikzpicture}
\begin{loglogaxis}[condcmp, title={multiscale $P$-greedy}, ymin=1e-2, ymax=1e12]
\addplot[msc] table[skip first n=1, x=n, y=cond] {\abstrace{ms_pgreedy}};
\addplot[gray, dashed] table[skip first n=1, x=n, y expr=\ckappaM*\thisrowno{8}] {\abstrace{ms_pgreedy}};
\addplot[gray, dotted, line width=1pt] table[skip first n=1, x=n, y expr=\czeroM*\thisrowno{7}] {\abstrace{ms_pgreedy}};
\end{loglogaxis}
\end{tikzpicture}%
\begin{tikzpicture}
\begin{loglogaxis}[condcmp, title={CC $P$-greedy}, ymin=1e-2, ymax=1e12]
\addplot[stc] table[skip first n=1, x=n, y=cond] {\abstrace{stable_pgreedy}};
\addplot[gray, dashed] table[skip first n=1, x=n, y expr=\ckappaM*\thisrowno{8}] {\abstrace{stable_pgreedy}};
\addplot[gray, dotted, line width=1pt] table[skip first n=1, x=n, y expr=\czeroM*\thisrowno{7}] {\abstrace{stable_pgreedy}};
\end{loglogaxis}
\end{tikzpicture}
\caption{Approximation of \eqref{eq:abs_function}: condition number of every kernel matrix built by the six greedy methods against the number of centres, with the two bounds of \eqref{eq:general_conditioning} evaluated for that matrix, the upper one dashed and the lower one dotted.
Top: $f$-greedy rule. Bottom: $P$-greedy rule. Left to right: single scale, multiscale with prescribed level sizes, conditioning-controlled variant.}
\label{fig:abs_comparison_cond}
\end{figure}
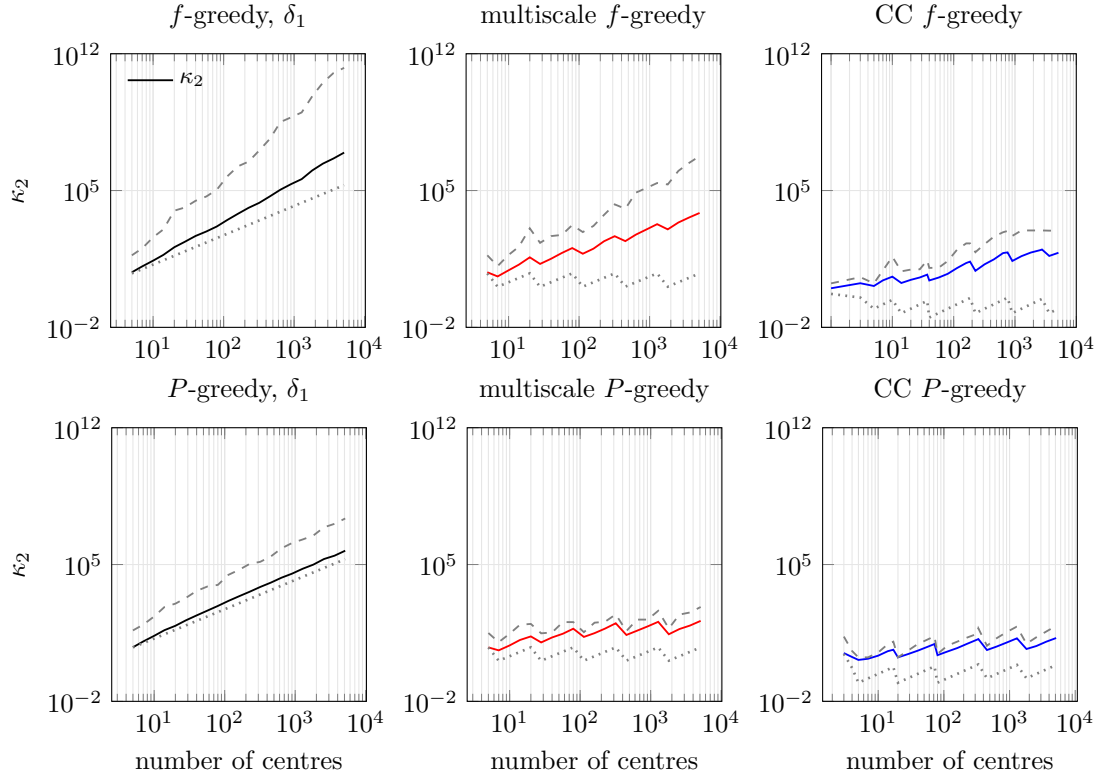

\paragraph{Convergence.}
\Cref{fig:abs_comparison_error} shows $\varepsilon_\infty$ and $\varepsilon_2$ against the number of centres.
The dashed lines are the {reference rates} of \Cref{th:stability_convergence}, which for $\tau = 3/2$ and $d = 2$ are $n^{-3/4}$ for the $f$-greedy rule and $n^{-1/4}$ for quasi-uniform centres in $\varepsilon_\infty$, and $n^{-3/4}$ for both in $\varepsilon_2$.
The final values are collected in \Cref{tab:abs_comparison}.

\begin{figure}[!htbp]
\centering
\begin{tikzpicture}
\begin{loglogaxis}[greedycmp, xlabel={}, ylabel={$\varepsilon_\infty$}, title={$f$-greedy},
                   legend pos=south west, ymin=1e-5, ymax=1e1]
\addplot[black] table[skip first n=1, x=n, y=linf] {\abstrace{fgreedy_1}}; \addlegendentry{single scale}
\addplot[msc] table[skip first n=1, x=n, y=linf] {\abstrace{ms_fgreedy}}; \addlegendentry{multiscale}
\addplot[stc] table[skip first n=1, x=n, y=linf] {\abstrace{stable_fgreedy}}; \addlegendentry{CC multiscale}
\addplot[gray, dashed, domain=5:5120, samples=2] {0.0894*x^(-0.75)}; \addlegendentry{$n^{-\tau/d}$}
\end{loglogaxis}
\end{tikzpicture}%
\begin{tikzpicture}
\begin{loglogaxis}[greedycmp, xlabel={}, ylabel={$\varepsilon_\infty$}, title={$P$-greedy},
                   legend pos=south west, ymin=1e-5, ymax=1e1]
\addplot[black] table[skip first n=1, x=n, y=linf] {\abstrace{pgreedy_1}}; \addlegendentry{single scale}
\addplot[msc] table[skip first n=1, x=n, y=linf] {\abstrace{ms_pgreedy}}; \addlegendentry{multiscale}
\addplot[stc] table[skip first n=1, x=n, y=linf] {\abstrace{stable_pgreedy}}; \addlegendentry{CC multiscale}
\addplot[clc] table[skip first n=1, x=n, y=linf] {\abstrace{classical}}; \addlegendentry{classical, Wendland}
\addplot[gray, dashed, domain=5:55120, samples=2] {0.0700*x^(-0.25)}; \addlegendentry{$n^{-\tau/d+1/2}$}
\end{loglogaxis}
\end{tikzpicture}\\
\begin{tikzpicture}
\begin{loglogaxis}[greedycmp, ylabel={$\varepsilon_2$},
                   legend pos=south west, ymin=1e-5, ymax=1e1]
\addplot[black, forget plot] table[skip first n=1, x=n, y=l2] {\abstrace{fgreedy_1}};
\addplot[msc, forget plot] table[skip first n=1, x=n, y=l2] {\abstrace{ms_fgreedy}};
\addplot[stc, forget plot] table[skip first n=1, x=n, y=l2] {\abstrace{stable_fgreedy}};
\addplot[gray, dashed, domain=5:5120, samples=2] {0.115*x^(-0.75)}; \addlegendentry{$n^{-\tau/d}$}
\end{loglogaxis}
\end{tikzpicture}%
\begin{tikzpicture}
\begin{loglogaxis}[greedycmp, ylabel={$\varepsilon_2$},
                   legend pos=south west, ymin=1e-5, ymax=1e1]
\addplot[black, forget plot] table[skip first n=1, x=n, y=l2] {\abstrace{pgreedy_1}};
\addplot[msc, forget plot] table[skip first n=1, x=n, y=l2] {\abstrace{ms_pgreedy}};
\addplot[stc, forget plot] table[skip first n=1, x=n, y=l2] {\abstrace{stable_pgreedy}};
\addplot[clc, forget plot] table[skip first n=1, x=n, y=l2] {\abstrace{classical}};
\addplot[gray, dashed, domain=5:55120, samples=2] {0.3*x^(-0.75)}; \addlegendentry{$n^{-\tau/d}$}
\end{loglogaxis}
\end{tikzpicture}
\caption{Approximation of \eqref{eq:abs_function}: decay of $\varepsilon_\infty$ (top) and $\varepsilon_2$ (bottom) against the number of centres. Left: $f$-greedy rule. Right: $P$-greedy rule. The dashed lines are the {reference rates} of \Cref{th:stability_convergence}.}
\label{fig:abs_comparison_error}
\end{figure}
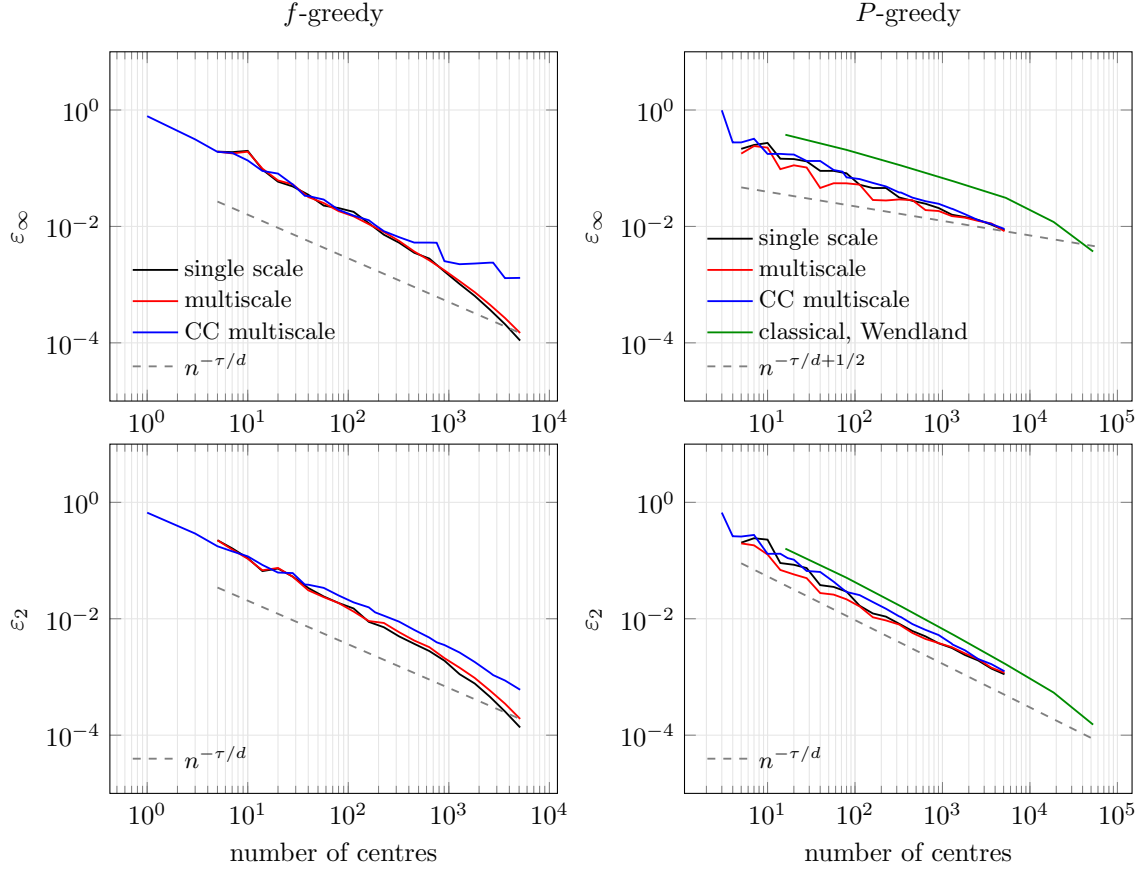

\paragraph{Cost.}
\Cref{fig:abs_comparison_cost} shows the offline time against the number of centres and the error against the offline time, and \Cref{tab:abs_comparison} collects the final state of every method.
The offline time of the greedy methods grows like $N n^2$, as predicted in \Cref{tab:cost}, and is the same for the single-scale and the multiscale algorithms, since with nested levels the greedy steps of all the levels add up to one single-scale run.
Only the CC $P$-greedy variant takes twice as long, because its certificate needs the residual on the candidates at every step.
The classical method grows linearly in the number of centres and has the smallest online cost, since its kernels are compactly supported.

\begin{table}[htb]
\centering
\renewcommand{\arraystretch}{1.25}
\small
\begin{tabular}{ l c c c c c c c c }
\hline
method & $L$ & $n$ & terms & $\varepsilon_\infty$ & $\varepsilon_2$ & $\max_\ell \kappa_2$ & offline [s] & online [s] \\
\hline
$f$-greedy, $\delta_1$   & 1 & $5\,120$  & $5\,120$  & $1.09 \cdot 10^{-4}$ & $1.36 \cdot 10^{-4}$ & $8.73 \cdot 10^{6}$ & $258$ & $11.9$ \\
multiscale $f$-greedy    & 6 & $5\,120$  & $6\,825$  & $1.48 \cdot 10^{-4}$ & $1.90 \cdot 10^{-4}$ & $7.13 \cdot 10^{3}$ & $264$ & $12.2$ \\
CC $f$-greedy        & 8 & $5\,120$  & $8\,872$  & $1.31 \cdot 10^{-3}$ & $6.03 \cdot 10^{-4}$ & $9.61 \cdot 10^{1}$ & $271$ & $17.1$ \\
$P$-greedy, $\delta_1$   & 1 & $5\,120$  & $5\,120$  & $8.61 \cdot 10^{-3}$ & $1.11 \cdot 10^{-3}$ & $5.10 \cdot 10^{5}$ & $256$ & $8.6$ \\
multiscale $P$-greedy    & 6 & $5\,120$  & $6\,825$  & $8.27 \cdot 10^{-3}$ & $1.18 \cdot 10^{-3}$ & $1.31 \cdot 10^{2}$ & $279$ & $13.2$ \\
CC $P$-greedy        & 7 & $5\,120$  & $6\,846$  & $9.04 \cdot 10^{-3}$ & $1.25 \cdot 10^{-3}$ & $1.73 \cdot 10^{1}$ & $510$ & $16.9$ \\
classical   & 5 & $5\,297$  & $7\,089$  & $3.10 \cdot 10^{-2}$ & $1.65 \cdot 10^{-3}$ & $3.96 \cdot 10^{2}$ & $0.9$ & $1.3$ \\
classical & 7 & $52\,791$ & $78\,558$ & $3.71 \cdot 10^{-3}$ & $1.52 \cdot 10^{-4}$ & $3.96 \cdot 10^{2}$ & $2.2$ & $2.2$ \\
\hline
\end{tabular}
\caption{Approximation of \eqref{eq:abs_function}: state of every method at the final level $L$. $n$ is the number of centres, terms the number of kernel terms of the approximant, $\max_\ell \kappa_2$ the largest condition number over the levels. The offline time is the time of the selection and of the solution of the level systems, the online time is the time to evaluate the final approximant at the $N = 100\,000$ data points.}
\label{tab:abs_comparison}
\end{table}

\begin{figure}[!htbp]
\centering
\begin{tikzpicture}
\begin{loglogaxis}[greedycmp, ylabel={offline time [s]}, legend style={draw=none}, width=0.42\textwidth, height=0.35\textwidth, legend pos = south east]
\addplot[black,forget plot] table[skip first n=1, x=n, y index=11] {\abstrace{fgreedy_1}};
\addplot[msc,forget plot] table[skip first n=1, x=n, y index=11] {\abstrace{ms_fgreedy}};
\addplot[stc,forget plot] table[skip first n=1, x=n, y index=11] {\abstrace{stable_fgreedy}};
\addplot[black, dashed,forget plot] table[skip first n=1, x=n, y index=11] {\abstrace{pgreedy_1}};
\addplot[msc, dashed, forget plot] table[skip first n=1, x=n, y index=11] {\abstrace{ms_pgreedy}};
\addplot[stc, dashed, forget plot] table[skip first n=1, x=n, y index=11] {\abstrace{stable_pgreedy}};
\addplot[clc, forget plot] table[skip first n=1, x=n, y index=11] {\abstrace{classical}};
\addplot[gray, dashed, line width=0.6pt, domain=16:5120, samples=2] {9.82e-6*x^2};
\addlegendentry{$n^2$}
\addplot[gray, dotted, domain=16:52791, samples=2] {4.13e-5*x};
\addlegendentry{$n$}
\end{loglogaxis}
\end{tikzpicture}%
\begin{tikzpicture}
\begin{loglogaxis}[greedycmp,
  xlabel={offline time [s]}, ylabel={$\varepsilon_\infty$},
  ymin=1e-5, ymax=2e0,
  width=0.42\textwidth, height=0.35\textwidth,
  legend pos=outer north east,
  legend style={font=\footnotesize, /tikz/every even column/.append style={column sep=0.4em}}]
\addplot[black] table[skip first n=1, x index=11, y=linf] {\abstrace{fgreedy_1}}; \addlegendentry{$f$-greedy}
\addplot[msc] table[skip first n=1, x index=11, y=linf] {\abstrace{ms_fgreedy}}; \addlegendentry{multiscale $f$-greedy}
\addplot[stc] table[skip first n=1, x index=11, y=linf] {\abstrace{stable_fgreedy}}; \addlegendentry{CC $f$-greedy}
\addplot[black, dashed] table[skip first n=1, x index=11, y=linf] {\abstrace{pgreedy_1}}; \addlegendentry{$P$-greedy}
\addplot[msc, dashed] table[skip first n=1, x index=11, y=linf] {\abstrace{ms_pgreedy}}; \addlegendentry{multiscale $P$-greedy}
\addplot[stc, dashed] table[skip first n=1, x index=11, y=linf] {\abstrace{stable_pgreedy}}; \addlegendentry{CC $P$-greedy}
\addplot[clc] table[skip first n=1, x index=11, y=linf] {\abstrace{classical}}; \addlegendentry{classical, Wendland}
\end{loglogaxis}
\end{tikzpicture}
\caption{Approximation of \eqref{eq:abs_function}: offline time against the number of centres (left) and $\varepsilon_\infty$ against the offline time (right) for all the methods, with $N = 100\,000$ candidates. The $P$-greedy rule is dashed.}
\label{fig:abs_comparison_cost}
\end{figure}
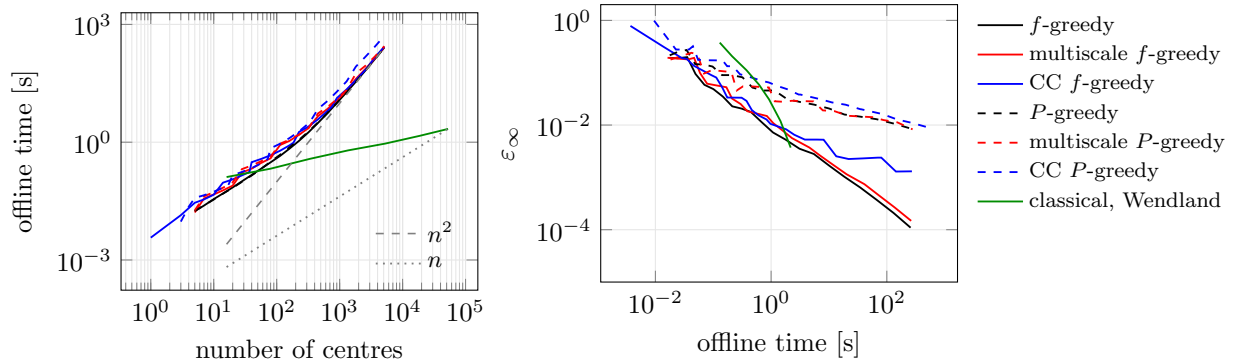

\section{Conclusion and outlook}
\label{sec:conclusion}

We introduced greedy multiscale kernel algorithms, in which every level of a multiscale residual correction is built by a greedy kernel algorithm.
For the multiscale $f$-greedy algorithm the convergence rate in the number of centres is that of the single-scale algorithm, and the cost of the selection is the same.
What the multiscale structure changes is the stability.
The condition number of a kernel matrix is bounded from below by a power of $\chi = \delta n^{1/d}$, for any position of the centres, 
and the $f$-greedy error bound is a negative power of the same $\chi$: at a fixed scale, accuracy has to be paid for with conditioning.
The multiscale scaling keeps $\chi$ fixed on every level, so that each level provides only a bounded reduction, but at a condition number that no longer has to grow with the number of centres.
\Cref{th:stability_convergence} states this for any construction of the levels, with the reduction of a level deciding the convergence and the separation of its centres deciding the stability.
The numerical experiments confirm the picture, in error, time and conditioning.

We conclude by mentioning two interesting directions for future research.
First, our analysis of the multiscale $f$-greedy algorithm assumes that $f$ lies in the native space.
Extending it to targets outside the native space requires a corresponding result for the single-scale $f$-greedy algorithm, which is still open, and the experiment with the kink in \Cref{subsec:comparison} suggests that the algorithms perform well in this regime.
Second, evaluating the selection criterion at all candidates on every level costs $\mathcal{O}(N M_n^2)$.
Restricting the candidates to the cells of an adaptive samplet tree \cite{avesani2025multiscale} would reduce this cost to linear in $N$, 
up to logarithmic factors, which we consider a promising direction for future work.

\textbf{Acknowledgements:}
The work was initiated during a research stay of Tizian Wenzel at USI Lugano. Sara Avesani and Michael Multerer gratefully acknowledge support from the Swiss National Science Foundation (SNSF) through the
SNSF starting grant ``Multiresolution methods for unstructured data'' (TMSGI2\_211684).

\bibliography{references}
\bibliographystyle{abbrv}

\end{document}